\documentclass[11pt]{amsart}
\usepackage{amsfonts, amsmath}
\usepackage{color}
\usepackage[mathscr]{eucal}
\usepackage{amscd, latexsym, graphicx, mathrsfs, enumerate}
\usepackage{amssymb}

\makeindex

\newtheorem{thm}{{{Theorem}}}[section]
\newtheorem{prop}[thm]{{Proposition}}
\newtheorem{lem}[thm]{{Lemma}}
\newtheorem{cor}[thm]{{Corollary}}

\newtheorem{remark}[thm]{Remark}
\numberwithin{equation}{section}

\def\N{\mathbb{N}}
\def\Z{\mathbb{Z}}
\def\Q{\mathbb{Q}}
\def\R{\mathbb{R}}
\def\C{\mathbb{C}}
\def\A{\mathbb{A}}
\def\GL{{\mathop{\mathrm{GL}}}}
\def\PGL{{\mathop{\mathrm{PGL}}}}
\def\SL{{\mathop{\mathrm{SL}}}}
\def\O{{\mathop{\mathrm{O}}}}

\def\U{{\mathop{\mathrm{U}}}}

\def\Re{{\mathop{\mathrm{Re}}}}

\def\Tr{{\mathop{\mathrm{Tr}}}}

\def\diag{{\mathop{\mathrm{diag}}}}
\def\sgn{{\mathop{\mathrm{sgn}}}}
 
\def\vol{{\mathop{\mathrm{vol}}}}

\def\d{{\mathrm{d}}}

\def\bsl{\backslash}
\def\inf{\infty}
\def\fin{\mathrm{fin}}

\def\fa{{\mathfrak{a}}}

\def\trep{{\mathbf{1}}}

\def\fO{{\mathfrak{O}}}

\def\cS{{\mathcal{S}}}
\def\ff{{\mathfrak{f}}}
\def\fd{{\mathfrak{d}}}

\def\us{{\underline{s}}}
\def\um{{\underline{m}}}
\def\ut{{\underline{t}}}

\def\cs{{\mathfrak{c}_0}}

\def\fD{{\mathfrak{D}}}

\numberwithin{equation}{section}

\title[The Shintani double zeta functions]
{The Shintani double zeta functions}
\author{Henry H. Kim, Masao Tsuzuki, and Satoshi Wakatsuki}
\date{\today}
\thanks{The first author is partially supported by NSERC grant \#482564. The second and third authors are partially supported by JSPS Grant-in-Aid for Scientific Research (C) No. 15K04795, 18K03235.}
\address{Henry H. Kim \\
Department of mathematics \\
 University of Toronto \\
Toronto, Ontario M5S 2E4, CANADA \\
and Korea Institute for Advanced Study, Seoul, KOREA}
\email{henrykim@math.toronto.edu}

\address{Masao Tsuzuki \\
Department of Science and Technology \\
Sophia University \\
Kioi-cho 7-1 Chiyoda-ku Tokyo, JAPAN}
\email{m-tsuduk@sophia.ac.jp}

\address{Satoshi Wakatsuki \\
Faculty of Mathematics and Physics, Institute of Science and Engineering\\
Kanazawa University\\
Kakumamachi, Kanazawa, Ishikawa, 920-1192, JAPAN}
\email{wakatsuk@staff.kanazawa-u.ac.jp}

\begin{document}

\begin{abstract} 
In this paper, we give an explicit formula of the Shintani double zeta functions with any ramification in the most general setting of adeles over an arbitrary number field. Three applications of the explicit formula are given. First, we obtain a functional equation satisfied by the Shintani double zeta functions
in addition to Shintani's functional equations. 
Second, we establish the holomorphicity of a certain Dirichlet series generalizing a result by Ibukiyama and Saito. This Dirichlet series occurs in the study of unipotent contributions of the geometric side of the Arthur-Selberg trace formula of the symplectic group. 
Third, we prove an asymptotic formula of the weighted average of the central values of quadratic Dirichlet $L$-functions. 
\end{abstract}

\maketitle

\setcounter{tocdepth}{1}
\tableofcontents

\section{Introduction}
For positive integers $m$ and $n$, let $A(m,n)$ denote the number of distinct solutions to the quadratic congruence equation $x^2\equiv n \mod m$. In his original work \cite{Shintani}, Takuro Shintani introduced Dirichlet series with two complex variables 
$\underline s=(s_1,s_2)\in \C^2$
\begin{equation}\label{intro1}
\xi_i(\us) :=\sum_{m=1}^\inf \sum_{n=1}^\inf \frac{ A(4m,(-1)^{i-1} n) }{m^{s_1}n^{s_2}},\qquad \xi_i^*(\us) :=\sum_{m=1}^\inf \sum_{n=1}^\inf \frac{ A(m,(-1)^{i-1} n) }{m^{s_1}(4n)^{s_2}} \quad (i=1,2)
\end{equation}
and showed that \begin{align*}
&\Gamma\left(\tfrac {s_1+1}2\right)^{-1}s_1(2s_1-1)\zeta(2s_1)(s_2-1)(s_1-1)^2(2s_1+2s_2-3) \times \xi_i(\us),\\ 
&\Gamma\left(\tfrac {s_1+1}2\right)^{-1}s_1(2s_1-1)\zeta(2s_1)(s_2-1)(s_1-1)^2(2s_1+2s_2-3) \times \xi_i^*(\us)
\end{align*}
have analytic continuations to $\C^2$ satisfying a set of functional equations. By an explicit determination of the singular part of \eqref{intro1}, he obtained an asymptotic formula for an average of the class numbers of quadratic fields with growing discriminants. Now the Dirichlet series \eqref{intro1} are commonly referred to as the {\it Shintani double zeta functions}. 
 In later works (Diamantis-Goldfeld \cite{DG}, Ibukiyama-Saito \cite{IS2}), it is discovered that there are relations between the Shintani double zeta-functions \eqref{intro1} and the Mellin transforms of metaplectic Eisenstein series on $\GL(2)$, or the {\it $A_2$ Weyl group multiple Dirichlet series} ($A_2$-WMDS for short). The latter object inherits two functional equations, one from the $A_2$-Weyl group action on the spectral parameter of the Eisenstein series and the other from the automorphy of the Eisenstein series. The relation between the Shintani double zeta function and the $A_2$-WMDS is nicely put in a work by J. Wen (\cite[Theorem 2.7]{Wen}), where the odd discriminant part of $\xi_1(\us)+\xi_2(\us)$ is identified with an $A_2$-WMDS up to a simple factor explicitly written by the Riemann zeta function. Thus, the Shintani double zeta functions satisfy not only Shintani's functional equations \cite[Theorem 1]{Shintani}, but also a different type of functional equations as an $A_2$-WMDS. 

In this paper, we study an adelic version of the Shintani double zeta functions over any number field in a general setting with possible ramifications; we obtain the expected functional equation (Theorem~\ref{thm:funct2}) for them as prehomogeneous zeta functions, working in the framework of Shintani and employing several known results from \cite{Sato3} and \cite{Shintani}. We derive an explicit formula 
(Theorem \ref{thm:global}) for the Shintani double zeta functions in a form of multiple Dirichlet series of quadratic Hecke $L$-functions, which provides us with yet another functional equation for them (Theorem~\ref{thm:funct1}). Several applications of the explicit formula will be given. 

Let us explain our results in more detail. For simplicity, we restrict ourselves to the rational number field $\Q$ in the introduction. The algebraic group
\begin{equation}\label{intro2}
G:=\GL(1)\times \left\{ \begin{pmatrix}1&0 \\ *&* \end{pmatrix} \in \GL(2) \right\}
\end{equation}
over $\Q$ acts on the $\Q$-vector space
\begin{equation}\label{intro3}
V:= \{ x\in M(2) \mid x={}^t\!x \} \quad \text{via} \quad (a,h)\cdot x=\rho(a,h)x:=ahx{}^t\!h \qquad ((a,h)\in G).
\end{equation}
This pair $(G,V)$ is a prehomogeneous vector space, i.e., there is an open dense $G$-orbit in $V$. Basic relative invariants $P_1$ and $P_2$ on $V$ are given by
\begin{equation}\label{intro4}
P_1(x):=x_1,\quad P(x)=P_2(x):=x_{12}^2-x_1x_2=-\det(x) \quad \text{for $x=\left(\begin{smallmatrix}x_1&x_{12} \\ x_{12}&x_2 \end{smallmatrix}\right) \in V$}.
\end{equation}
We set
\begin{equation}\label{intro5}
V^0:=\{x\in V \mid P_1(x)\neq 0 , \;\; P_2(x)\neq 0\}.
\end{equation}
Let $S$ be a finite set of places of $\Q$ such that $\inf\in S$ and set $\Q_S:=\prod_{v\in S}\Q_v$. For $x=(x_v)_{v\in S} \in \Q_S^\times$, set $|x|_S:=\prod_{v\in S}|x_v|_v$ with $|\; |_v$ being the valuation of $\Q_v$. Fix an element $\delta_S\in \Q_S^\times / (\Q_S^\times)^2$ for a while. Depending on $S$ and $\delta_S$, we have a discrete set $L_S(\delta_S):=V(\Q)\cap (V^0(\Q_S,\delta_S)\prod_{p\not\in S}V(\Z_p))$ endowed with the natural action of the $S$-unit group $\Gamma_S:= G(\Q)\cap (G(\Q_S)\prod_{p\not\in S}G(\Z_p))$ of $G$, where $V^0(\Q_S,\delta_S):=\{x\in V^0(\Q_S)\mid P_2(x)\in \delta_S(\Q_S^\times)^2\}$.
According to the usual manner of prehomogeneous vector spaces, we define the prehomogeneous zeta function $\xi^S(\us,\delta_S)$ by
\begin{equation}\label{intro6}
\xi^S(\us,\delta_S):= \sum_{x\in\Gamma_S\bsl L_S(\delta_S)} \frac{1}{\#(\Gamma_{S,x})}\frac{1}{|P_1(x)|_S^{s_1}\, |P_2(x)|_S^{s_2}},
\end{equation}
where $\Gamma_{S,x}:=\{ \gamma\in \Gamma_S \mid \gamma\cdot x=x \}$. The series \eqref{intro6} is absolutely convergent for $\Re(s_1)>1$ and $\Re(s_2)>1$.
If $S=\{\inf\}$, then we have $\xi^S(\us,1)=2^{2s_2-1} \xi_1^*(\us)$ and $\xi^S(\us,-1)=2^{2s_2-1} \xi_2^*(\us)$. Hence, $\xi^S(\us,\delta_S)$ is viewed as a natural generalization of the Shintani double zeta functions. To explain our explicit formula of $\xi^S(\us,\delta_S)$, we need some additional notations. Let $\A$ denote the adele ring of $\Q$. In the usual manner, a Dirichlet character is identified with a character of $\Q^\times \R_{>0}\bsl \A^\times$. For each character $\chi=\otimes_v \chi_v$ of $\Q^\times\R_{>0}\bsl \A^\times$, let $L^S(s,\chi)$ denote the Dirichlet $L$-function without the $S$-factors defined as the partial Euler product $L^S(s,\chi):=\prod_{p\not\in S}L_p(s,\chi_p)$, where $L_p(s,\chi_v):=(1-\chi_p(p)p^{-s})^{-1}$ if $\chi_p$ is unramified, and $L_p(s,\chi_p):=1$ if $\chi_p$ is ramified. Note that $\zeta^S(s):=L^S(s,{\bf 1})$ coincides with the Riemann zeta function without the $S$-factors. Let $f_\chi=\prod_{p<\infty}p^{f(\chi_p)} \in \Z_{>0}$ be the conductor of $\chi$, and  $f_\chi^S:=\prod_{p\not\in S}p^{f(\chi_p)}$ its prime-to-$S$ part. Let $\widehat{\Q_S^\times/( \Q_S^\times)^2}$ denote the set of real valued characters of $\Q_S^\times$. Note that $[\Q_S^\times: (\Q_S^\times)^2]<\infty$. 

For any finite set $S$ of places of $\Q$ such that $\inf\in S$, define 
$\tilde \xi^S(\us,\omega_S)$ by
\begin{equation*}
\tilde \xi^S(\us,\omega_S):=\frac{2^{|S|+1}}{|\Bbb Q_S^\times/(\Bbb Q_S^\times)^2|\, |2|_S} \sum_{\delta_S \in \Bbb Q_S^\times/(\Bbb Q_S^\times)^2} \omega_S(\delta_S) \times \xi^S(\us,\delta_S), \quad 
\omega_S\in \widehat{\Bbb Q_S^\times/(\Bbb Q_S^\times)^2}
.
\end{equation*}

The following theorem (Theorem \ref{thm:global}) is our main result.
\begin{thm}\label{thm:main} For any $\omega_S\in \widehat {\Q_S^\times/ (\Q_S^\times)^2}$, 
\[
\tilde \xi^S(\us,\omega_S)=\zeta^S(s_1)\, \zeta^S(2s_1+2s_2-1) \sum_{\chi} \frac{L^S(s_2,\chi)}{L^S(2s_1+s_2,\chi)\, (f_\chi^S)^{s_1}}
\]
with $\chi$ moving over all real valued characters $\chi=\otimes_v \chi_v$ of $\Q^\times\R_{>0}\bsl \A^\times$ such that $\otimes_{v\in S}\chi_v=\omega_S$.
\end{thm}

A proof of this theorem will be given in \S\ref{sec:explicit}. Following the method of Shintani (\cite{Shintani}), we eventually show that $\xi^S(\us,\delta_S)$ is meromorphically continued to $\C^2$ in \S\ref{sec:prin}.
In particular, from this theorem we obtain an explicit formula of $\xi_j^*(\us)$ as
\[
\xi^*_j(\us)= 2^{-2s_2-1} \zeta(s_1)  \, \zeta(2s_1+2s_2-1) \sum_{D} \frac{ \sgn(D)^{j-1} \, L(s_2,\chi_D)  }{ L(2s_1+s_2,\chi_D)\, |D|^{s_1}}
\]
where $D$ moves over $1$ and all fundamental discriminants and $\chi_D$ denotes the Dirichlet character on $\Z/D\Z$ corresponding to $\Q(\sqrt{D})$.
It follows from this explicit formula and the functional equation of $L(s_2,\chi_D)$ that one gets the new functional equation
\begin{multline*}
\frac{\zeta(2s_1+2s_2-1)}{\zeta(s_1+s_2-\frac{1}{2})} \left( \xi_1^*(s_1+s_2-\frac{1}{2},1-s_2) +(-1)^k \xi_2^*(s_1+s_2-\frac{1}{2},1-s_2) \right) \\
=   2^{3s_2-1}\pi^{-s_2}\Gamma(s_2) \frac{\zeta(2s_1)}{\zeta(s_1)} \, \left( \xi^*_1(\us)+(-1)^k \xi^*_2(\us) \right) \begin{cases} \cos(\pi s_2/2)  & \text{if $k=0$,} \\  \sin(\pi s_2/2)  & \text{if $k=1$.} \end{cases}
\end{multline*}
We should remark that Ibukiyama and Saito (\cite[p.291, Section 3.3]{IS2}) also proved an explicit formula for $\xi_j(\us)$ in a non-adelic setting, and Taniguchi (\cite[Proposition B.9]{Taniguchi}) gave an adelic explicit formula; their formulas are similar to but curiously different from ours in that the same quadratic $L$-functions are incorporated in but the variables $s_1$ and $s_2$ are interchanged in their explicit formula; the difference seems to be crucial because their formula itself cannot provide other functional equations than Shintani's functional equations. (See Remark \ref{rem:explicit} for detail.)

To simplify our explanation on functional equations, we further suppose $\{\infty,2\} \subset S$. We obtain a functional equation of $\tilde\xi^S(\us,\omega_S)$ relating the values at $(s_1+s_2-\frac{1}{2},1-s_2)$ and $(s_1,s_2)$ (Theorem~\ref{thm:funct1}) from the familiar functional equation of the partial Dirichlet $L$-functions : $L^S(1-s,\chi)=N(\ff_{\chi}^S)^{s-1/2} \, \Gamma_S(s,\chi_S) \, L^S(s,\chi)$, where $\Gamma_S(s,\chi_S)$ is a gamma factor (see \eqref{eq:gammaS}). Furthermore, we also obtain a functional equation of $\tilde\xi^S(\us,\omega_S)$ relating the values at $(s_1,\frac{3}{2}-s_1-s_2)$ and $(s_1,s_2)$ (Theorem~\ref{thm:funct2}) from the local functional equations of Shintani \cite{Shintani} and Sato \cite{Sato}. It seems natural to expect that the Shintani double zeta functions $\tilde\xi^S(\us,\omega_S)$ possess a group of functional equations isomorphic to $D_{12}$, the dihedral group of order $12$.
This looks like an analogue of the functional equations of a double Dirichlet series introduced by Blomer \cite{B}, see also \cite{FF}.

For a positive integer $m$ and a character $\omega_S$ of $\Q_S^\times/(\Q_S^\times)^2$, consider the one-variable zeta function
\begin{align}
D_m(s,\omega_S):=\frac {\zeta^S(2s-m+1)}{\zeta^S(s-\frac m2+\frac 12)} \, \tilde\xi^S\left( (s-\frac m2+\frac 12,\frac m2),\omega_S\right), \quad s\in \C.
 \label{Dmsomega}
\end{align}
When $m=2$, this function was studied in \cite{Datskovsky} and also in \cite{HW}. For some time, it has been observed that the zeta function $D_m(s,\omega_S)$ occurs in the unipotent contributions of the geometric side of the Arthur-Selberg trace formula of the symplectic group $\mathrm{Sp}(m)$ (see \cite[Section 6]{IS2} and \cite[Section 1.2]{Wakatsuki}). Indeed, in their series of works (\cite{IS1, IS2}), Ibukiyama and Saito explicitly computed the central unipotent contribution to the dimension formulas for the Siegel modular forms in terms of the Bernoulli numbers. In doing so, they studied $D_m(s,\omega_S)$ with $S=\{\infty\}$ for any $m$ and proved that $D_m(s,\omega_S)$ is holomorphic at non-positive integers (\cite[Proposition 3.6]{IS2}).
We generalize this holomorphicity result to arbitrary $S$ in Corollary \ref{cor:applicationtotarceformula} by using Theorem \ref{thm:main} and Shintani's method for zeta integrals. (See \eqref{eq:globallocal} for the relation between $\xi^S(\us,\delta_S)$ and the global zeta integral.) This result on holomorphicity is of crucial importance in the forthcoming work \cite{KWY}, where an equidistribution theorem in the level aspect for the Satake parameters of holomorphic Siegel cusp forms of general degree is proved.

By the zeta integral method, it is shown that $D_1(s,\omega_S)$ has a double pole at $s=1$ (Proposition \ref{prop:simplepole3.17}).
Let 
$$D_1(s,\omega_S)=\text{(some factor)}\times \sum_{N\in \mathfrak{N}(\omega_S)} H(1/2,N,\omega_S)N^{-s},
$$
where $\mathfrak{N}(\omega_S)$ is the set of positive integers $N=(-1)^{\delta_\omega} Df^2$ such that $D$ is a fundamental discriminant 
with $\chi_{D,v}=\omega_v$ for all $v\in S$. Here $H(1/2,N,\omega_S)$ is the generalized Cohen function (\ref{cohen}).
In particular, if $N$ is a fundamental discriminant, $H(1/2,N,\omega_S)=L(1/2,\chi_{N})$.

If we know that $H(1/2,N,\omega_S)$ is non-negative, we can use Sato and Shintani's generalization of Landau's theorem \cite[Theorem 3]{SS}, and prove an asymptotic formula 
$$\sum_{N\leq x,\, N\in \mathfrak{N}(\omega_S)} H(1/2,N,\omega_S)= Ax\log x+Bx+O(x^{1/3}),
$$
for some constants $A,B$. However, it is not proved yet that $H(1/2,N,\omega_S)$ is non-negative. We will prove, for any $\epsilon>0$, there exist constants $A,B$ such that 
$$\sum_{N\leq x,\, N\in \mathfrak{N}(\omega_S)} H(1/2,N,\omega_S)= Ax\log x+Bx+O(x^{19/32+\epsilon}).
$$

Let us explain the structure of this paper briefly. In Section \ref{prem}, we review the Tate integral in both local and global settings, and reintroduce the prehomogeneous vector space $(G,V)$ together with some other additional objects. In Section \ref{local-zeta}, a basic results on the local zeta integrals attached to our prehomogeneous vector space is briefly recalled. In Section \ref{double-zeta}, the global zeta integral $Z(\Phi,\us)$ of two complex variables $\us=(s_1,s_2)$ is introduced in the setting of adeles, and the meromorphic continuation and the functional equation is proved. In Theorem \ref{thm:global}, we relate $Z(\Phi,\us)$ to the Shintani double zeta function $\tilde \xi^S(\us,\omega_S)$. In Section \ref{fun}, we prove functional equations of $\tilde \xi^S(\us,\omega_S)$.
Finally, in Section \ref{appl}, we give the application on asymptotics of central values of quadratic Dirichlet $L$-functions.

\medskip

\noindent\textbf{Acknowledgments.} We would like to thank Shuichi Hayashida, Chihiro Hiramoto, Tomoyoshi Ibukiyama, Fumihiro Sato, Takashi Taniguchi, Takuya Yamauchi, and Akihiko Yukie for helpful discussions.

\section{Preliminaries}\label{prem}
\subsection{Notations}\label{sec:notations}
Let $F$ be an algebraic number field.
Let $\Sigma_\inf$ (resp. $\Sigma_\fin$) denote the set of all the infinite (resp. finite) places of $F$.
The set $\Sigma=\Sigma_\inf\cup\Sigma_\fin$ consists of all the places of $F$.
Let $\Sigma_\infty=\Sigma_\Bbb R\cup\Sigma_\C$, where $\Sigma_\R$ (resp. $\Sigma_\mathbb{C}$) denote the set of real (resp. complex) places of $F$.
We also set $\Sigma_2:=\{ v\in\Sigma_\fin \mid v$ divides $2 \}$.
For any $v\in\Sigma$, we denote by $F_v$ the completion of $F$ at $v$. 
For each $v\in\Sigma_\fin$, we denote by $\fO_v$ the ring of integers of $F_v$.
Let $\pi_v$ be a prime element of $\fO_v$. We put $q_v=\#(\fO_v/\pi_v\fO_v)$.
Let $\A$ denote the adele ring of $F$ and $\A_\fin$ the finite adele ring of $F$.

Let $\d x$ denote the Haar measure on $\A$ normalized by $\int_{\A/F} \d x=1$. 
For each $v\in\Sigma_\R$ we write $\d x_v$ for the ordinary Lebesgue measure on $\R$, and for each $v\in\Sigma_\C$ we set $\d x_v:=2\, \d x_{v,1} \d x_{v,2}$ for $x_v=x_{v,1}+x_{v,2}i\in \C$ where $\d x_{v,1}$ and $\d x_{v,2}$ denote the Lebesgue measure on $\R$. For $v\in \Sigma_\fin$, we fix a Haar measure on $F_v$ normalized by $\int_{\fO_v}\d x_v=1$. Then it is known that
\begin{align}
\d x = \Delta_F^{-1/2} \prod_{v\in\Sigma} \d x_v 
 \label{HaarMesDec}
\end{align}
holds, where $\Delta_F$ denotes the absolute discriminant of $F/\Q$.
We denote by $|\; |_v$ the normal valuation of $F_v$.
Then, we have $\d (ax_v)=|a|_v \d x_v$ for any $a\in F_v^\times$.
We define the idele norm $|\; |=|\; |_\A$ on $\A^\times$ by $|x|=|x|_\A=\prod_{v\in\Sigma} |x_v|_v$ for all $x=(x_v)\in\A^\times$.

We fix a non-trivial additive character $\psi_\Q$ on $\A_\Q/\Q$ such that $\psi(x)=e^{2\pi ix}$ for $x\in \Q_\infty=\R$, and set $\psi_F=\psi_\Q \circ \Tr_{F/\Q}$. Then, $\d x$ is the self-dual Haar measure with respect to $\psi_F$. For $v\in \Sigma$, set $\psi_{F_v}=\psi_{F}|F_v$. 
Then $\psi_{F_v}(x)=e^{2\pi i x}$ for $v\in \Sigma_\R$ and $\psi_{F_v}(x)=e^{4\pi i \Re(x)}$ for $v\in \Sigma_\C$, so that $d x_v$ is self-dual with respect to $\psi_{F_v}$ for all $v\in\Sigma_\inf$. For each $v\in\Sigma_\fin$, $\fd_v$ denotes the differential exponent of $F_v/\Q_p$ where $p$ is the residual characteristic of $F_v$, i.e., $\pi_v^{-\fd_v}\fO_v=\{\xi\in F_v\mid {\rm Tr}_{F_v/\Q_p}(\xi\fO_v) \subset \Z_p\}$. Thus $q_v^{-\fd_v/2}\d x_v$ is the self-dual measure on $F_v$ with respect to $\psi_{F_v}$, and $\Delta_F=\prod_{v\in\Sigma_\fin}q_v^{\fd_v}$. Let $\d^\times x_v$ denote a Haar measure on $F_v^\times$ for each $v\in\Sigma$ defines as ${\displaystyle \d^\times x_v=(1-q_v^{-1})^{-1}\, \tfrac{\d x_v}{|x|_v} }$ if $v\in\Sigma_\fin$ and ${\displaystyle \d^\times x_v=\tfrac{\d x_v}{|x|_v} }$ if $v\in \Sigma_\infty$. Then $\int_{\fO_v^\times}\d^\times x_v=1$ for $v\in \Sigma_\fin$. The idele norm $|\;|$ induces an isomorphism $\A^\times/\A^1\to\R_{>0}$.
We choose the Haar measure $\d^\times x=\prod_{v\in\Sigma}\d^\times x_v$
on $\A^\times$ and normalize the Haar measure $\d^1 x$ on $\A^1$ in such a way
that the quotient measure on $\R_{>0}$ is $\d t/t$, where $\d t$ is the
Lebesgue measure on~$\R$.

Let $S$ be a finite subset of $\Sigma$.
We set $F_S=\prod_{v\in S}F_v$.
We define the norm $|\; |_S$ on $F_S^\times$ (resp. an additive character $\psi_{F_S}$ of $F_S$) by $|x|_S=\prod_{v\in S}|x_v|_v$ for $x=(x_v)\in F_S^\times$ (resp. $\psi_{F_S}(x)=\prod_{v\in S}\psi_{F_v}(x_v)$ for $x=(x_v)\in F_S$).
For any vector space $V$ over $F$, let $\cS(V(F_S))$, $\cS(V(\A))$, and $\cS(V(\A_\fin))$ denote the Schwartz spaces of $V(F_S)$, $V(\A)$, and $V(\A_\fin)$, respectively. We use the notation $F_\inf$ in place of $F_{\Sigma_\inf}$.

\subsection{Tate integral (Local)}\label{sec:TateLocal}
Let $v\in \Sigma$, and let $\mathbf 1_v$ be the trivial character of $F_v^\times$. 
For $\phi_v\in \cS(F_v)$ and a character $\chi_v:F_v^\times\rightarrow \Bbb T$, the local Tate integral 
\[
\zeta_v(\phi_v,s,\chi_v):=\int_{F_v^\times} \phi_v(a)\, |a|_v^s \,  \chi_v(a) \, \d^\times a
\]
is absolutely convergent for $\Re(s)>0$, and is meromorphically continued to the whole $s$-plane in such a way that $L_v(s,\chi_v)\,\zeta_v(\phi_v,s,\chi_v)$ gets entire, where 
\[
L_v(s,\chi_v):=\begin{cases} (1-\chi_v(\pi_v)q_v^{-s})^{-1}& \text{if $\chi_v$ is unramified,} \\ 1 & \text{if $\chi_v$ is ramified,} \end{cases}.
\]
For our purpose, we only need those $\chi_v$ such that $\chi_v^2={\bf 1}_v$, i.e., $\chi_v$ is a character of $F_v^\times/(F_v^\times)^2$ and real valued.
For such an $\chi_v$, the local functional equation of local Tate integral takes the form
\[
\zeta_v(\hat\phi_v,s,\chi_v) = \tilde\gamma_v(s,\chi_v)\, \zeta_v(\phi_v,1-s,\chi_v) \qquad (\phi_v\in\cS(F_v))
\]
where
\[
\hat\phi_v(y):=\int_{F_v} \psi_{F_v}(xy) \, \phi_v(x) \, \d x, \quad y\in F_v
\]
is the Fourier transform of $\phi_v$ and $\tilde\gamma_v(s,\chi_v)$ is the local gamma factor explicitly given as
\[
\tilde\gamma_v(s,\trep_v)=(2\pi)^{1-2s} \, \Gamma(s)/ \Gamma(1-s)   \qquad \text{for $v\in\Sigma_\C$,}
\]
\[
\tilde\gamma_v(s,\sgn^\delta)= i^\delta \, \pi^{\frac{1}{2}-s} \, \Gamma\left(\tfrac{s+\delta}{2}\right)/ \Gamma\left(\tfrac{1-s+\delta}{2}\right) \qquad \text{for $v\in\Sigma_\R$ $(\delta=0$ or $1)$,} 
\]
for $v\in\Sigma_\fin$
\[
\tilde\gamma_v(s,\chi_v)=q_v^{\fd_v s}  \times \begin{cases}
 (1-\chi_v(\pi_v)q_v^{-1+s})/ (1-\chi_v(\pi_v) q_v^{-s}) & \text{if $\chi_v$ is unramified,} \\
g_{\chi_v} \, N(\ff_{\chi_v})^s  & \text{if $\chi_v$ is ramified,}
\end{cases}
\]
where $\ff_{\chi_v}=\pi_v^{f_v}\fO_v$ denotes the conductor of $\chi_v$ and  
\[
g_{\chi_v}:=N(\ff_{\chi_v})^{-1}\sum_{u\in \fO_v^\times/(1+\ff_{\chi_v})}\chi_v(u) \, \psi_{F_v}(u\pi_v^{-\fd_v-f_v})
\]
is the Gauss sum for $\chi_v$.

We set
\[
\gamma_v(s,u):=\sum_{\chi_v\in\widehat{F_v^\times/(F_v^\times)^2}} \chi_v(u) \, \tilde\gamma_v(s,\chi_v) \quad (u\in F_v^\times).
\]
Then, from the functional equation of $\zeta_v(\phi,s,\chi_v)$, we can derive
\begin{equation}
\int_{\delta(F_v^\times)^2} \hat\phi(x)\, |x|_v^{s-1} \, \d x = \frac{1}{\#(F_v^\times/(F_v^\times)^2)}\sum_{\eta\in F_v^\times/(F_v^\times)^2} \gamma_{v}(s,\delta\eta) \, \int_{\eta(F_v^\times)^2} \phi(x)\, |x|_v^{-s} \, \d x
\end{equation}
for each $\delta\in F_v^\times$ and $\phi\in\cS(F_v)$. For convenience, we record an explicit formula of $\gamma_v(s,\delta)$ for $v\in\Sigma_\R$ and $\delta\in \{+1,-1\}$: 
\begin{align}
\gamma_v(s,1)=2(2\pi)^{-s} \, \Gamma(s) e^{i \pi s/2}, \qquad \gamma_v(s,-1)=2(2\pi)^{-s} \, \Gamma(s) e^{-i \pi s/2} .
 \label{realgammafac}
\end{align}

\subsection{Tate integral (Global)} \label{sec:TateGlobal}

The Hecke $L$-function of a character $\chi=\prod_v\chi_v$ of $\A^1/F^\times\cong \A^\times/\R_{>0}F^\times$ is defined as the absolutely convergent Euler product
\[
L(s,\chi)=\prod_{v<\inf}L_v(s,\chi_v), \quad \Re(s)>1. 
\]
It is well-known that $L(s,\chi)$ is meromorphically continued to the whole complex $s$-plane. Let $\mathbf 1_F$ be the trivial character of $\A^1/F^\times$. The function $L(s,\chi)$ is holomorphic except for a possible simple pole at $s=1$ which occurs if and only if $\chi=\mathbf 1_F$. We also set
\[
\zeta_F(s):=L(s,\mathbf 1_F), \quad c_F=\mathrm{Res}_{s=1}\zeta_F(s)=\vol(F^\times\bsl \A^1).
\]
For a finite subset $S\subset \Sigma_\fin$, we also consider the partial Euler products $L^{S}(s,\chi)=\prod_{v\not\in S}L_v(s,\chi_v)$ and $\zeta_{F}^{S}(s)=L^{S}(s,{\bf 1}_{F})$ without the $S$-factors. 

For $\phi\in\cS(\A)$, $s\in\C$, and a character $\chi$ on $\A^1/F^\times$, the global Tate integral is defined as
\[
\zeta(\phi,s,\chi)=\int_{\A^\times} \phi(x)\, |x|^s \, \chi(x) \, \d^\times x,
\]
which is known to be absolutely convergent for $\Re(s)>1$ and has a meromorphic continuation to $\C$ satisfying the functional equation
\begin{align}
\zeta(\hat\phi,s,\chi)=\zeta(\phi,1-s,\chi^{-1}),
 \label{TateFE}
\end{align}
where 
\[
\hat\phi(y)=\int_\A \psi_F(xy)\, \phi(x) \, \d x.
\]
is the Fourier transform of $\phi\in \cS(\A)$. We put $\zeta(\phi,s):=\zeta(\phi,s,\mathbf 1_F)$ for simplicity. As is well-known, the functional equation $\widehat\zeta_F(s)=\widehat\zeta_F(1-s)$ is deduced from \eqref{TateFE} by choosing a suitable function $\phi=\otimes_{v\in\Sigma} \phi_v$, where 
\[
\widehat\zeta_F(s):=\Delta_F^{s/2}\left\{ \pi^{-s/2}\Gamma(s/2) \right\}^{r_1} \left\{ (2\pi)^{-s}\Gamma(s) \right\}^{r_2} \, \zeta_F(s) 
\]
with $r_1=\#(\Sigma_\R)$ and $r_2=\#(\Sigma_\C)$ is the completed Dedekind zeta function of $F$. When $S$ contains $\Sigma_\inf \cup \{v\in\Sigma_\fin \mid \fd_v\neq 0\}$, we have the following asymmetric functional equation for the partial zeta function $\zeta_F^S(s)$:
\[
\zeta_F^S(1-s)=\Delta_F^{-1/2} \gamma_S(s) \, \zeta_F^S(s)
\]
with $\gamma_{S}(s)=\prod_{v\in S}\tilde\gamma_v(s,{\bf 1}_v)$, which becomes 
\begin{align}
\zeta_F(1-s)=\Delta_F^{s-1/2}\, \cos(s\pi/2)^{r_1+r_2} \, \sin(s\pi/2)^{r_2} \, (2(2\pi)^{-s}\Gamma(s))^{r_1+2r_2} \, \zeta_F(s)
 \label{AsymFT}
\end{align}
when $S=\Sigma_\infty$. 

Let $\chi=\otimes_{v\in\Sigma}\chi_v$ be a non-trivial quadratic character of $\A^1/F^\times$ and $E$ the quadratic extension over $F$ corresponding to $\chi$ by class field theory. Let $\ff_{\chi}$ be the conductor of $\chi$ and $N(\ff_{\chi})$ the absolute norm of $\ff_{\chi}$. We have $N(\ff_{\chi})=\prod_{v\in \Sigma_\fin} \#(\fO_v/\ff_{\chi_v})$, where $\ff_{\chi_v}$ denotes the conductor of $\chi_v$. Put $t=t(\chi) =\#\{v\in\Sigma_\R \mid \chi_v=\mathbf 1_v  \}$, so that the number of real places (resp. complex places) of $E$ is $2t$ (resp. $r_1-t+2r_2$). Then one can derive
\[
L(1-s,\chi)=N(\ff_{\chi})^{s-1/2} \, \Delta_F^{s-1/2}\, \cos(s\pi/2)^{t+r_2} \, \sin(s\pi/2)^{r_1-t+r_2} \, (2(2\pi)^{-s}\Gamma(s))^{r_1+2r_2} \, L(s,\chi)
\]
from the functional equations \eqref{AsymFT} for $F$ and $E$ in conjunction with the relations $\Delta_E=N(\ff_{\chi})\, \Delta_F^2$ and $\zeta_E(s)=\zeta_F(s)\, L(s,\chi)$. For any finite set of places $S$ containing $\Sigma_\inf$, one has
\begin{equation}\label{eq:functquad}
L^S(1-s,\chi)=N(\ff_{\chi}^S)^{s-1/2} \, \Gamma_S(s,\chi) \, L^S(s,\chi)
\end{equation}
where $N(\ff_{\chi}^S):=\prod_{v\in \Sigma_\fin \setminus S} \#(\fO_v/\ff_{\chi_v})$ and 
\begin{multline}\label{eq:gammaS}
 \Gamma_S(s,\chi) = \Delta_F^{s-1/2}\, \cos(s\pi/2)^{t+r_2} \, \sin(s\pi/2)^{r_1-t+r_2} \, (2(2\pi)^{-s}\Gamma(s))^{r_1+2r_2}  \\
\times \prod_{v\in S\cap\Sigma_\fin} N(\ff_{\chi_v})^{s-\frac{1}{2}}  \times \prod_{v\in S,\, \text{$\chi_v$ unramified}} \frac{1-\chi_v(\pi_v)q_v^{-1+s}}{1-\chi_v(\pi_v)q_v^{-s}}.
\end{multline}
In what follows, $\Gamma_S(s,\chi)$ which really depends only on $\chi_S=\prod_{v\in S}\chi_v$ will be denoted by $\Gamma_{S}(s,\chi_S)$. Note that since $\chi$ is quadratic, $\chi_S$ is viewed as a character of $F_S^\times/(F_S^\times)^2$. 


\subsection{Prehomogeneous vector space}

We recall the algebraic group $G$, the $F$-rational representation $(\rho,V)$ of $G$, and the basic relative invariants $P_1$ and $P_2=P$ defined by \eqref{intro2}, \eqref{intro3}, and \eqref{intro4}, respectively, which are viewed as objects defined over the field $F$. For notational simplification, we freely identify a matrix $\left(\begin{smallmatrix} x_1 & x_{12} \\ x_{12} & x_{2} \end{smallmatrix}\right)\in V$ with the 3-dimensional vector $(x_1,x_{12},x_{2})$. 
Let $\tau_j$ denote the $F$-rational character corresponding to the relative $G$-invariant polynomial $P_j$, i.e., $P_j(g\cdot x)=\tau_j(g)P_j(x)$ $(j=1,2)$ for all $x\in V,\,g\in G$. A computation reveals 
\[
\quad \tau_1(g)=a, \quad \tau_2(g)=a^2c^2 \quad \text{for} \quad  g=\left(a,\left(\begin{smallmatrix}1&0\\ b&c \end{smallmatrix}\right)\right)\in G.
\]
We identify $V$ with its dual space by the non-degenerate $F$-bilinear form
\[
\langle x,y\rangle:= x_1y_2-2x_{12}y_{12}+x_2y_1=\Tr(x \, (\det(y) \, y^{-1}))=\Tr(x \, JyJ^{-1}) \quad \text{where} \quad J=\left(\begin{smallmatrix}0&1 \\ -1&0 \end{smallmatrix}\right).
\]
Then the contragredient representation of $(\rho,V)$ is realized on the same space $V$ with the $G$-action $\hat \rho$ given by 
\[
\hat\rho(a,h) y:= a^{-2}\det(h)^{-2} \times \rho(a,h)\,y  \qquad \text{for $(a,h)\in G$},
\]
so that the relation $\langle \rho(g) x,\hat\rho(g) y\rangle=\langle x,y\rangle$ holds for all $g\in G$ and $x,\,y\in V$. It is confirmed that $P_j(\hat\rho(g) x)=\hat\tau_j(g)P_j(x)$ with $F$-rational characters $\hat\tau_j\,(j=1,2)$ given as 
\[\hat\tau_1(g)=a^{-1}c^{-2}, \qquad \hat\tau_2(g)=a^{-2}c^{-2}.
\]
The basic invariants $P_1$ and $P_2$ corresponds to the following constant coefficient differential operators on $V$:  
\[
D_{1,x}=-\frac{\partial}{\partial x_2} , \quad D_{2,x}=\frac{1}{4}\frac{\partial^2}{\partial x_{12}^2}-\frac{\partial^2}{\partial x_1\partial x_2}.
\]
In this setting, the $b$-function whose existence is ensured by a general theory can be explicitly determined as 
\begin{equation}\label{eq:bfun}
b_\um(\us)=[s_2+1]_{m_2} \, [s_1+s_2+3/2]_{m_1+m_2}
\end{equation}
which fit in the formula
\[
D_{1,x}^{-m_1}D_{2,x}^{m_1+m_2} (P_1(x)^{s_1+m_1}P_2(x)^{s_2+m_2})=b_\um(\us) \, P_1(x)^{s_1}P_2(x)^{s_2}.
\]
for all $\um=(m_1,m_2)\in \Z^2$, where we set $[\eta]_k:=\prod_{j=1}^k(\eta+j)$ if $k\geq 0$, and $[\eta]_k:=\prod_{j=k}^{-1}(\eta+j)^{-1}$ if $k< 0$. 

%

\section{Some results for local zeta functions}\label{local-zeta}

\subsection{Basic facts on the local zeta integral} \label{sec:BasicFactsLocalZeta}

Let $v\in \Sigma$. We choose a Haar measure $\d x$ on the space $V(F_v)$ as
\[
\d x:=\d x_1\, \d x_{12} \, \d x_2 \quad (x=(x_1,x_{12},x_2))
\]
where $\d x_*$ is the Haar measure on $F_v$ fixed in \S \ref{sec:notations}.
For $\delta\in F_v^\times$, we set 
$$V(F_v,\delta):= \{ x\in V(F_v) \mid P(x)\in \delta(F_v^\times)^2 \}, \quad V^0(F_v,\delta)=V^0(F_v)\cap V(F_v,\delta).
$$ 
For any square class $\delta\in F_v^\times/(F_v^\times)^2$, the local zeta integral is defined as
\[
Z_v(\Phi,\us,\delta):=\int_{V(F_v,\delta)} \Phi(x)\, |x_1|_v^{s_1-1} \, |P(x)|_v^{s_2-1} \, \d x, \quad \Phi \in \cS(V(F_v)). 
\]
For a character $\chi$ of $F_v^\times/(F_v^\times)^2$, we set
\[
\tilde Z_v(\Phi,\us,\chi):=\int_{V(F_v)}|x_1|_v^{s_1-1}|P(x)|_v^{s_2-1} \chi(P(x)) \Phi(x) \d x , \quad \Phi \in \cS(V(F_v )). 
\]
\begin{lem}\label{lem:2019.3.26conv}
The integrals $Z_v(\Phi,\us,\delta)$ and $\tilde Z_v(\Phi,\us,\chi)$ are absolutely convergent and holomorphic in the domain $\{\us\in\C^2\mid \Re(s_1)>\tfrac{1}{2}$, $\Re(s_2)>0$, $\Re(s_1+\tfrac{s_2}{2})>1\}$ (resp. $\{\us\in\C^2\mid \Re(s_1)>0$, $\Re(s_2)>0$, $\Re(s_1+s_2)>\tfrac{1}{2}\}$) if $v\in\Sigma_\inf$ (resp. $v\in\Sigma_\fin$).
Note that it is known that they are meromorphically continued to $\C^2$; see e.g., \cite{Sato3}.
\end{lem}
\begin{proof}
Let $v\in\Sigma_\inf$. It is sufficient to prove that 
\begin{equation}\label{eq:20200905}
\int_{V(F_v)}|x_1|_v^{t_1-1} \, |P(x)|_v^{t_2-1}  \, \phi_1(x_1) \, \phi_2(x_{12}) \phi_3(x_2) \, \d x_1 \, \d x_{12} \, \d x_2 <\inf 
\end{equation}
is convergent for $\phi_j\in\cS(F_v)$, $t_1$, $t_2\in\R$, $t_1>\tfrac{1}{2}$, $t_2>0$, $t_1+\tfrac{t_2}{2}>1$.
In order to prove this, we need 
\begin{equation*}\label{eq:20200905v1}
\int_{F_v} \phi(y+a)\, |y|_v^{t_2-1} \, \d y \ll_{\phi,M,t_2} 1\qquad  \text{if $0\leq |a| \leq M$ and $0<t_2$,} 
\end{equation*}
\begin{equation*}\label{eq:20200905v2}
\int_{F_v} \phi(y+a)\, |y|_v^{t_2-1} \, \d y \ll_{\phi,M} \begin{cases} 
1+|a|_v^{t_2-\tfrac{1}{2}} & \text{if $M\leq |a|$ and $1\leq t_2$}, \\   
1+|a|_v^{\tfrac{t_2}{2}} & \text{if $M\leq |a|$ and $0<t_2\leq 1$}
\end{cases}
\end{equation*}
for any positive constant $M$ and any test function $\phi\in\cS(F_v)$.
These inequalities can be proved by a direct calculation.
By change of variable $y=x_2-\tfrac{x_{12}^2}{x_1}$ 
we have
\[
\eqref{eq:20200905}=\int_{F_v}\int_{F_v}\int_{F_v} |x_1|^{t_1+t_2-2} |y|^{t_2-1} \, \phi_1(x_1) \, \phi_2(x_{12})\, \phi_3(y+\tfrac{x_{12}^2}{x_1}) \, \d x_1 \, \d x_{12} \, \d y.  
\]
Hence, we obtain the convergence range of \eqref{eq:20200905} by applying the inequalities to the above for $y$.

For $v\in\Sigma_\fin$, the assertion follows from the proof of Theorem \ref{thm:local}, since any compact domain in $V(F_v)$ is contained in $\pi_v^{-l}V(\fO_v)$ for some $l\in\N$.
\end{proof}

The following expressions of $Z_v(\Phi,\us,\delta)$ and $\tilde Z(\Phi,\us,\chi)$ as integrals over $G(F_v)$ will be needed later. 
\begin{lem} \label{lem:coordinate-expression}
 Set $\tilde \delta=\diag(1,-\delta)$. Then, 
\begin{align*}
Z_v(\Phi,\us,\delta)&=(1-q_v^{-1})^22^{-1}|2|_v\,|\delta|_v^{s_2}
\int_{F_v^\times} \int_{F_v}\int_{F_v^\times} |a|_v^{s_1}|a^2c^2|_v^{s_2}\Phi(a,ab,a(b^2-\delta c^2))\,\d^\times a\,\d b\,\d^\times c
\\
&=(1-q_v^{-1})^2 2^{-1}|2|_v\int_{G(F_v)}|P_1(g\,\tilde\delta)|_v^{s_1}|P_2(g\,\tilde\delta)|_v^{s_2}\Phi(g\,\tilde \delta)\,\d g, 
\end{align*}
and for any character $\chi$ of $F_v^\times/(F_v^\times)^2$, 
\begin{align*}
\tilde Z(\Phi,\us,\chi)=(1-q_v^{-1})^2 \int_{F_v^\times}\int_{F_v}\int_{F_v^\times} |a|_v^{s_1}|a^2c|_v^{s_2}\Phi(a,ab,a(b^2-c))\chi(c)\,\d^\times a\,\d b\,\d^\times c.
\end{align*}
\end{lem}
\begin{proof}
The set $V^0(F_v,\delta)$ is a single $G(F_v)$-orbit containing the diagonal matrix $\tilde\delta=\diag(1,-\delta)$; thus $V^0(F_v,\delta)=G(F_v)\,\tilde \delta$. It is easy to see that the natural map $j$ from $G(F_v)$ onto $V^0(F_v,\delta)$ defined as $j(g)=g\tilde \delta$ is two-to-one. Let $x=\left(\begin{smallmatrix} x_1 & x_{12} \\ x_{12} & x_2 \end{smallmatrix}\right)$ and $g=\left(a,\left(\begin{smallmatrix} 1 & 0 \\ b & c \end{smallmatrix}\right) \right)\in G(F_v)$ be related by $x=\rho(g)\,\tilde\delta$, or equivalently $x_1=a$, $x_{12}=ab$ and $x_2=a(b^2-\delta c^2)$. Then a computation yields $P_1(g\tilde \delta)=a$, $P_2(g\tilde \delta)=a^2c^2\delta$, and 
$$
j^{*}\d x=|c|_v|a|_v^2|2\delta|_v\,\d a\,\d b\,\d c=(1-q_v^{-1})^{2}|2\delta|_v\,|a|_v^3|c|^2_v\,\d^\times a\,\d b\,\d^\times c.    
$$
From these, we obtain the first formula immediately. By decomposing $F_v^\times$ to cosets $\delta(F_v^\times)^2$ and then substituting the formula shown above, we have\begin{align*}
\tilde Z(\Phi,\us,\chi)&=\sum_{\delta}\chi(\delta)\,Z(\Phi,\us,\delta)
\\
&=(1-q_v^{-1})^{2}2^{-1}|2|\sum_{\delta}
\int_{F_v^\times} \int_{F_v}\int_{F_v^\times} |a|_v^{s_1}|a^2c^2\delta|_v^{s_2}\chi(\delta c^2)\,\Phi(a,ab,a(b^2-\delta c^2))\,\d^\times a\,\d b\,\d^\times c.
\end{align*}
The map $c\mapsto c'=c^2\delta$ is a two-to-one surjection form $F_v^\times$ onto $\delta(F_v^\times)^2$ and $\d^\times c'=|2|_v\,\d^\times c$. Hence by the variable change $c'=c^2\delta$ in the last expression, the $\delta$-summation and the $c$-integral are combined to form $2|2|_v^{-1}$ times a $c'$-integral. Thus  \begin{align*}
Z(\Phi,\us,\chi)&=
(1-q_v^{-1})^{2}2^{-1}|2|_v \times 2|2|_v^{-1} \int_{F_v^\times}\int_{F_v}\int_{F_v^\times}|a|_v^{s_1}|a^2c'|_v^{s_2}\chi(c')\,\Phi(a,ab,a(b^2-c'))\,\d^\times a\,\d b\,\d^\times c'.
\end{align*} 
This proves the second formula. 
\end{proof}

\begin{lem}\label{lem:nonv}
For each $\us\in\C^2$ and each $\delta\in F_v^\times/(F_v^\times)^2$, there exists a test function $\Phi\in\cS(V(F_v))$ such that the support of $\Phi$ is contained in $V^0(F_v)$, $Z_v(\Phi,\us,\delta)$ is convergent, $Z_v(\Phi,\us,\delta)\neq 0$ and for any $\delta_1\in F^\times/(F^\times)^2$ with $\delta_1\not=\delta$ we have $Z_v(\Phi,\us,\delta_1)=0$.
\end{lem}
\begin{proof}
The family of sets $V^0(F_v,\delta)$ $(\delta\in F_v^\times/(F_v^\times)^2)$ is an open covering of $V^0(F_v)$ such that $V^0(F_v,{\delta}) \cap V^0(F_v,{\delta_1})=\emptyset$ if $\delta_1\not=\delta$. The integral $Z_v(\Phi,\us,\delta)$, which is clearly absolutely convergent for $\Phi\in \cS(V^0(F_v))$, defines a measure $|x_1|^{s_1-1}|P(x)|_v^{s_2-1}\d x$ on $V^0(F_v)$ with full support $V^0(F_v)$. Hence we can find $\Phi\in \cS(V^0(F_v))$ such that $Z_v(\Phi,\us,\delta)\not=0$ and $Z_v(\Phi,\us,\delta_1)=0$ for $\delta_1\not=\delta$.
\end{proof}

\subsection{Functional equations}
The Fourier transform of $\Phi\in \cS(V(F_v))$ is defined by 
\[
\hat\Phi(y):=\int_{V(F_v)} \Phi(x) \, \psi_{F_v}(\langle x,y\rangle)\, \d x \qquad (y\in V(F_v)),
\]
For $\delta,\xi\in F_v^\times/(F_v^\times)^2$, set
\begin{equation}\label{G}
G_v(\us,\delta,\xi)= \frac{1}{|2|_v^{1/2}\, \#(F_v^\times/(F_v^\times)^2)^2}  \sum_{\eta\in F_v^\times/(F_v^\times)^2} \alpha_{\psi_{F_v}}(-\eta) \, \gamma_{v}(s_2,\delta \eta) \, \gamma_v \left(s_1+s_2-\tfrac{1}{2},\eta\xi\right) ,
\end{equation}
where $\alpha_{\psi_{F_v}}(-\eta)$ is the Weil constant so defined that the relation
\[
\int_{F_v} \phi(x)\,  \psi_{F_v}(a x^2) \, \d x = \alpha_{\psi_{F_v}}(a)\, |2a|_v^{-1/2} \int_{F_v} \hat\phi(x) \, \psi_{F_v}\left(-\tfrac{x^2}{4a}\right)\, \d x
\]
holds for any $\phi\in\cS(F_v)$, which shows that $\alpha_{\psi_{F_v}}(a)$ depends only on the square class $a(F_v^\times)^2$ of $a$. (See \cite{Ik}.) For $v\in\Sigma_\R$, we have $F_v^\times/(F_v^\times)^2=\{+1,-1\}$ and $\alpha_{\psi_{F_v}}(+1)=e^{\pi i /4}$ and $\alpha_{\psi_{F_v}}(-1)=e^{-\pi i /4}$; this combined with \eqref{realgammafac} immediately yields the explicit formulas 
\begin{multline}\label{eq:gamma4.10}
\begin{pmatrix}G_v(\us,1,1) & G_v(\us,1,-1) \\ G_v(\us,-1,1) & G_v(\us,-1,-1) \end{pmatrix}=\\
2^{1-s_1-2s_2}\pi^{\frac{1}{2}-s_1-2s_2}\Gamma(s_2)\Gamma\left(s_1+s_2-\tfrac{1}{2}\right) \begin{pmatrix} \sin\pi(\frac{s_1}{2}+s_2) & \cos\pi s_1/2 \\ \sin \pi s_1/2 & \cos\pi(\frac{s_1}{2}+s_2) \end{pmatrix}. 
\end{multline}

 
\begin{lem}\label{lem:Rfe} Let $v\in \Sigma$. Then for any $\Phi\in\cS(V(F_v))$ and any $\us=(s_1,s_2)\in\C^2$, one obtains the functional equation
\[
Z_v(\hat\Phi,\us,\delta)=\sum_{\xi\in F_v^\times/(F_v^\times)^2} G_v(\us,\delta,\xi) \, Z_v\left(\Phi,(s_1,\tfrac{3}{2}-s_1-s_2),\xi\right) \qquad (\delta\in F_v^\times/(F_v^\times)^2).
\]
\end{lem}
\begin{proof}
See \cite{Sato3}. See also \cite[Lemma 1 (i)]{Shintani} for $v\in\Sigma_\R$.
\end{proof}
From Lemma \ref{lem:Rfe}, we immediately obtain the local functional equations of $\tilde Z_v(\Phi,\us,\chi)$ as in the next lemma. 
\begin{lem}\label{lem:localfechar}
\[
\tilde Z_v(\hat\Phi,\us,\chi) = \sum_{\omega\in\widehat{F_v^\times/(F_v^\times)^2}} \tilde {G}_v(\us,\chi,\omega) \,\tilde {Z}_v\left(\Phi,(s_1,\tfrac{3}{2}-s_1-s_2\right),\omega),\]
where
\begin{equation}
\tilde G_v(\us,\chi,\omega):=\frac{1}{|2|^{1/2}\, \#(F_v^\times/(F_v^\times)^2)}\, \tilde \gamma_v(s_2,\chi) \, \tilde \gamma_v\left(s_1+s_2-\tfrac{1}{2},\omega\right) \sum_{\eta\in F_v^\times/(F_v^\times)^2} \alpha_{\psi_{F_v}}(-\eta) \, \chi\omega(\eta) 
 \label{GG}
\end{equation}
for $\chi$, $\omega\in\widehat{F_v^\times/(F_v^\times)^2}$.
\end{lem}

For $v\in\Sigma_\R$, the characters of $F_v^\times/(F_v^\times)^2$ are ${\bf 1}$ and $\sgn$; from \eqref{realgammafac} and we can easily obtain the explicit formulas 
\begin{multline}\label{eq:funct2019March25}
\begin{pmatrix}\tilde G_v(\us,\trep,\trep) & \tilde G_v(\us,\trep,\sgn) \\ \tilde G_v(\us,\sgn,\trep) & \tilde G_v(\us,\sgn,\sgn) \end{pmatrix}=2^{\frac{3}{2}-s_1-2s_2}\pi^{\frac{1}{2}-s_1-2s_2}\Gamma(s_2)\Gamma\left(s_1+s_2-\tfrac{1}{2}\right)\\
\times \begin{pmatrix} \cos(\pi s_2/2)\cos\frac{\pi}{2}(s_1+s_2-\frac{1}{2}) & \cos(\pi s_2/2)\sin\frac{\pi}{2}(s_1+s_2-\frac{1}{2}) \\ \sin(\pi s_2/2)\cos\frac{\pi}{2}(s_1+s_2-\frac{1}{2}) & -\sin(\pi s_2/2)\sin\frac{\pi}{2}(s_1+s_2-\frac{1}{2}) \end{pmatrix}. 
\end{multline}

\subsection{Non-vanishing}\label{subsec:poleslocal}
Set
\begin{equation}\label{eq:K}
K_v:=\begin{cases}   \U(2) & \text{if $v\in\Sigma_\C$,}  \\ \O(2) & \text{if $v\in\Sigma_\R$,}  \\   \GL(2,\fO_v) & \text{if $v\in\Sigma_\fin$.}  \\\end{cases}
\end{equation}
endowed with a Haar measure $\d k$ such that $\int_{K_v}\d k=1$. We say that $\Phi \in \cS(V(F_v))$ is $K_v$-spherical if $\Phi(kx{}^t\!k)=\Phi(x)$ holds for any $k\in K_v$. We quote several results on archimedean local zeta integrals for $K_v$-spherical test functions in the following three lemmas for later use.
\begin{lem}\label{lem:snonv1}
Let $v\in\Sigma_\C$.
In this case, $Z_v(\Phi,\us,\delta)$ does not depend on the choice of $\delta$.
Choose a $K_v$-spherical test function $\Phi$ as in \cite[Theorem 6.3.1]{Igusa} such that $\hat\Phi=|2|_v^{-1/2}\, \Phi$. 
Then, one obtains
\[
|2|_v^{1/2} \, Z_v(\Phi,\us,\delta)=|2|_v \, Z_v(\hat\Phi,\us,\delta)=
2^{a_1s_1+a_2s_2+a_3}\pi^{b_1s_1+b_2s_2+b_3} \,  \Gamma(s_2) \, \Gamma\left(s_1+s_2-\tfrac{1}{2}\right)
\]
for some rational numbers $a_i, b_i$, $i=1,2,3$.
\end{lem}
\begin{proof}
This can be proved by using the explicit form \eqref{eq:bfun} of $b_\um(\us)$ and the argument in \cite[Proof of Theorem 6.3.1]{Igusa}.
\end{proof}
\begin{lem}\label{lem:Rfcteq}
{\rm \cite[Lemma 1 (ii)]{Shintani} and \cite[Lemma 2.9]{Sato2}}
Let $v\in\Sigma_\R$ and suppose $\Phi$ is $K_v$-spherical. Then 
\[
\frac{\sin(\pi s_1/2)}{\Gamma(s_2) \, \Gamma(s_1+s_2-\frac{1}{2})} Z_v(\Phi,\us,+1) ,\quad \frac{1}{\Gamma(s_2) \, \Gamma(s_1+s_2-\frac{1}{2})} Z_v(\Phi,\us,-1)
\]
are holomorphic on $\C^2$.
Furthermore, we have
\[
Z_v(\Phi,s_1,s_2,+1)=\frac{\cos(\pi s_1/2)}{\sin(\pi s_1/2)} \, Z_v\left(\Phi,1-s_1,s_1+s_2-\tfrac{1}{2},+1\right),
\]
\[
Z_v(\Phi,s_1,s_2,-1)= Z_v\left(\Phi,1-s_1,s_1+s_2-\tfrac{1}{2},-1\right).
\]
\end{lem}
\begin{lem}\label{lem:snonv2}
{\rm \cite[Remark of Lemma 1]{Shintani}}
Let $v\in\Sigma_\R$.
For each $\us\in\C^2$, there exist compactly supported $K_v$-spherical test functions $\Psi_1$, $\Psi_2\in\cS(V(F_v))$ such that $\Gamma(\frac{s_1+1}{2})Z_v(\Psi_1,\us,+1)\neq 0$, $Z_v(\Psi_2,\us,-1)\neq0$, and the support of $\Psi_j$ is included in $V(F_v,(-1)^{j-1})$.
\end{lem}
\begin{proof}
If $\Re(s_1)\geq 0$, then one can prove this assertion by choosing a support and calculating the integral directly.
As for $\Re(s_1)<0$, it is sufficient to apply Lemma \ref{lem:Rfcteq}.
\end{proof}

\begin{lem}\label{lem:polelocalp}
{\rm \cite[Proposition 2.4]{Sato3}}
Let $v\in\Sigma_\fin$.
There exists a positive integer $\cs$ such that
\[
(1-q^{-2s_1})^\cs \, (1-q^{-2s_2})^\cs \, (1-q^{-2s_1-2s_2+1})^\cs \, Z_v(\Phi,\us,\delta)  
\]
is holomorphic on $\C^2$ for any $\Phi\in \cS(V(F_v))$ and $\delta\in F_v^\times/(F_v^\times)^2$.
\end{lem}

\begin{lem}\label{lem:padicnonva}
Let $v\in\Sigma_\fin$.
For each $\delta\in F_v^\times/(F_v^\times)^2$, there exists a $K_v$-spherical test function $\Psi\in\cS(V(F_v))$ such that, $(1-q^{-s_1})Z_v(\Psi,\us,\delta)$ is holomorphic on $\C^2$, $(1+q^{-s_1})^{-1}Z_v(\Psi,\us,\delta)\neq0$, and the support of $\Psi$ is included in $V^0(F_v,\delta)$.
\end{lem}
\begin{proof}
This follows from Lemma \ref{lem:localAp12}.
\end{proof}

\subsection{Explicit formula of local zeta functions over $p$-adic fields}

We shall explicitly compute the local zeta functions over $p$-adic fields for some special test functions.
Let $v\in\Sigma_\fin$.
To simplify notation, we abbreviate $F_v$, $q_v$, $\pi_v$, $\chi_v$, $\fO_v$ to $F$, $q$, $\pi$, $\chi$, $\fO$ omitting the subscripts $v$ throughout this subsection.
\begin{thm}\label{thm:local} 
Let $v\in\Sigma_\fin$ and $\Phi_0$ the characteristic function of $V(\fO)$. Then for any character $\chi$ of $F^\times/(F^\times)^2$, we have
\[
(1-q^{-1})^{-2}\tilde Z_v(\Phi_{v,0},\us,\chi)=\frac{L_v(s_1,\mathbf 1) \,  L_v(2s_1+2s_2-1,\mathbf 1) \,  L_v(s_2,\chi)}{ L_v(2s_1+s_2,\chi)\, N(\ff_{\chi})^{s_1} },
\]
where $\mathbf 1$ denotes the trivial character of $F^\times$ and $\ff_{\chi}$ denotes the conductor of $\chi$.
\end{thm}
\begin{proof}
We start with the following formula in Lemma~\ref{lem:coordinate-expression}. 
\begin{equation}\label{eq:lc}
\tilde Z_v(\Phi,\us,\chi_v)= \int_{F_v}\int_{F_v}\int_{F_v} |a|_v^{s_1+2s_2-1}|c|_v^{s_2-1} \chi_v(c)\, \Phi(a,ab,ab^2-ac)\, \d a\, \d c \, \d b.
\end{equation}
From \eqref{eq:lc}, one can easily deduce
\begin{multline}\label{eq:1}
(1-q^{-1})^{-2}\tilde Z_v(\Phi_0,\us,\chi)=(1-q^{-1})\sum_{l\geq 0}\sum_{l+m\geq 0}\sum_{k\in \Z} \int_{\fO^\times}\d^\times b \, \int_{\fO^\times}\d^\times c \\
q^{-l(s_1+2s_2)} \, q^{-m} \,  (\chi(\pi)q^{-s_2})^k \, \chi(c) \times \phi(\pi^{l+2m}b^2-\pi^{l+k}c)
\end{multline}
where $\phi$ denotes the characteristic function of $\fO$.
In order to calculate \eqref{eq:1} explicitly, we divide the above total sum into the three partial sums according to the cases
\[
\text{(i) $ k < 2m$}, \quad \text{(ii) $k>2m$}, \quad \text{(iii) $k=2m$}.
\]

\noindent
{\bf Cases (i) and (ii).}
If $\chi$ is ramified, then the partial sums of (i) and (ii) vanish. Suppose that $\chi$ is unramified. Then,
\[
\eqref{eq:1}=(1-q^{-1})\sum_{l\geq 0}\sum_{l+m\geq 0}\sum_{k\in \Z} \int_{\fO^\times}\d^\times b \, \int_{\fO^\times}\d^\times c \;\; q^{-l(s_1+2s_2)} \, q^{-m} \,  (\chi(\pi)q^{-s_2})^k \, \phi(\pi^{l+2m}b^2-\pi^{l+k}c).
\]
For any $b$ and $c$ in $\fO^\times$, we also have $\phi(\pi^{l+2m}b^2-\pi^{l+k}c)=1$ if and only if $-l\leq k$ in the case (i), and $\phi(\pi^{l+2m}b^2-\pi^{l+k}c)=1$ if and only if $l+2m\geq 0$ in the case (ii).
By a direct calculation, one can prove that the partial sum for (i) is
\begin{equation}\label{eq:(i)}
\frac{1}{1-\chi(\pi)q^{-s_2}} \times\frac{1+\chi(\pi)q^{-s_1-s_2}}{1-q^{-2s_1-2s_2+1}} - \frac{1}{1-\chi(\pi)q^{-s_2}}  \times \frac{1-q^{-1}}{1-q^{-2s_2-1}}  \times \frac{1+q^{-s_1-2s_2}}{1-q^{-2s_1-2s_2+1}}.
\end{equation}
One can also show that the partial sum for (ii) is
\begin{equation}\label{eq:(ii)}
\frac{1}{1-\chi(\pi)q^{-s_2}} \times \frac{(1- q^{-1}) \, \chi(\pi) q^{-s_2} }{1-q^{-2s_2-1}} \times \frac{1+q^{-s_1-2s_2}}{1-q^{-2s_1-2s_2+1}} .
\end{equation}

\noindent
{\bf Case (iii).}
In this case, one gets $q^{-m} \,  (\chi(\pi)q^{-s_2})^k=q^{-m(2s_2+1)}$.
We further divide the case (iii) $k=2m$ into the two cases
\[
\text{(iii-1) $2m+l\geq 0$} \quad \text{and}  \quad  \text{(iii-2) $2m+l<0$}.
\]
In the case (iii-1), one has
\[
\int_{\fO^\times}\phi(\pi^{2m+l} (b^2-c))\, \chi(c)\, \d^\times c=\begin{cases}1&\text{if $\chi$ is unramified,} \\ 0 & \text{if $\chi$ is ramified.} \end{cases} 
\]
Hence, the partial sum for (iii-1) is
\[
\frac{1-q^{-1}}{1-q^{-2s_2-1}}\times \frac{1+q^{-s_1-2s_2}}{1-q^{-2s_1-2s_2+1}}
\]
if $\chi$ is unramified, and it vanishes if $\chi$ is ramified.

We shall consider the case (iii-2).
There exists a non-negative integer $f$ such that $\pi^f\fO$ is the conductor of $\chi$, i.e., $f=0$ if and only if $\chi$ is unramified; if $f\geq 1$, one has $\chi|_{1+\pi^f\fO}=1$ and $\chi|_{1+\pi^{f-1}\fO}\neq 1$.
By means of the evaluation
\begin{equation}\label{eq:ram}
(1-q^{-1})\int_{\fO^\times}\phi(\pi^{2m+l} (b^2-c))\, \chi(c)\, \d^\times c=  \int_{1+\pi^{-2m-l}\fO}  \chi(c) \d c = \begin{cases} q^{2m+l} & \text{if $-2m-l\geq f$}, \\ 0 & \text{if $-2m-l< f$,} \end{cases}
\end{equation}
the partial sum for (iii-2) is computed to be
\[
\frac{q^{-fs_1}}{1-q^{-s_1}}\times\frac{1}{1-q^{-2s_1-2s_2+1}}
\]
for $f\geq 1$. 
From \eqref{eq:ram}, a similar consideration for the case $f=0$ yields the same result as that of the case $f=1$ for the partial sum for (iii-2), i.e., for $f=0$, the partial sum for (iii-2) is equal to 
\[
\frac{q^{-s_1}}{1-q^{-s_1}}\times\frac{1}{1-q^{-2s_1-2s_2+1}}.
\]
Summarizing the above computations, one can complete the proof.
\end{proof}

As a corollary to this theorem, the local zeta integral $Z_v(\Phi_{v,0},\us,\delta)$ at a non-dyadic place is obtained as  
\begin{cor} \label{cor:local} Suppose $v\in \Sigma_\fin-\Sigma_2$. For each $\delta\in F^\times$, one has 
\begin{align*}
(1-q^{-1})^{-2}Z_v(\Phi_{v,0},\us,\delta)=\frac{1}{2}\, \frac{L_v(s_1,\mathbf 1)\, L_v(2s_2,\mathbf 1)\,  L_v(2s_1+2s_2-1,\mathbf 1) \, L_v(s_1, \chi_{\delta}) }{ L_v(2s_1,\mathbf 1) \, L_v(s_1+2s_2,\chi_{\delta})\, N(\ff_{\chi_{\delta}})^{s_2} } ,
\end{align*}
where we set $\chi_{\delta}(x):=( \delta , x)_v$ $(x\in F^\times)$ by the Hilbert symbol $(\; , \; )_v$ on $F^\times\times F^\times$.
\end{cor}
\begin{proof}
Since $v\not\in \Sigma_2$, we have $\widehat{F^\times/ (F^\times)^2}=\{{\bf 1}, \chi_1,\chi_2,\chi_1\chi_2\}$, where $\chi_1$ and $\chi_2$ are characters of $F^\times$ of order $2$ defined by the relations $\chi_1(\pi)=-1$, $\chi_1(u)=1\,(u\in \fO^\times)$ and $\chi_2(\pi)=\chi_2(u^2)=+1$ ($u\in \fO^\times$), $\chi_2(u)=-1$ ($u\in \fO^\times-(\fO^\times)^2)$. Obviously, $\ff_{{\bf 1}}=\ff_{\chi_{1}}=\fO$ and $\ff_{\chi_2}=\ff_{\chi_1\chi_2}=\pi\fO$. In particular, $L_v(s,\chi_2)=L_v(s,\chi_1\chi_2)=1$. Using these facts and substituting the formula in Theorem~\ref{thm:local} to 
$$
Z_v(\Phi,\us,\delta)=\frac{1}{4}\sum_{\chi\in \widehat{F^\times/(F^\times)^2}} \chi(\delta)\times  \tilde Z_v(\Phi,\us,\chi),
$$ 
we get the required formula after an easy direct computation. 
\end{proof}

To prove Lemma \ref{lem:padicnonva}, we need the following.
\begin{lem}\label{lem:localAp12} 
Choose an element $\delta\in \fO^\times\sqcup \pi \fO^\times$.
Let $\Psi$ denote the characteristic function of $K_v \cdot(\left(\begin{smallmatrix}1&0 \\ 0&-\delta \end{smallmatrix}\right)+\pi^2V(\fO))$ (resp. $K_v \cdot(\left(\begin{smallmatrix}0&1 \\ 1&0 \end{smallmatrix}\right)+\pi^2V(\fO))$) when $v\not\in\Sigma_2$ or $\delta\not\in (\fO^\times)^2$  (resp. $v\in\Sigma_2$ and $\delta\in (\fO^\times)^2$).
Then, one has
\[
Z_v(\Psi,\us,\delta)= c \times |\delta|^{s_2-1} \times \begin{cases}  1 & \text{if $\delta\in\fO^\times\setminus (\fO^\times)^2$,} \\  1+q^{-s_1} & \text{if $\delta\in\pi\fO^\times$,}  \\  (1+q^{-s_1})(1-q^{-s_1})^{-1} & \text{if $\delta\in (\fO^\times)^2$ and $v\not\in\Sigma_2$,}  \\  q^{-s_1}(1+q^{-s_1})(1-q^{-s_1})^{-1} & \text{if $\delta\in (\fO^\times)^2$, $v\in\Sigma_2$, and $2\in\pi\fO^\times$,}  \\  q^{-2s_1}(1-q^{-s_1})^{-1} & \text{if $\delta\in (\fO^\times)^2$, $v\in\Sigma_2$, and $2\in\pi^2\fO$,} \end{cases}
\]
for some constant $c$.
\end{lem}
\begin{proof}
First, let us consider the case $v\not\in\Sigma_2$ or $\delta\not\in (\fO^\times)^2$.
By a direct calculation one has
\[
Z_v(\Psi,\us,\delta)=\frac{q^{-6}|\delta|^{s_2-1}}{\sharp(H_\delta(\fO/\pi^2\fO))} \sum_{\left(\begin{smallmatrix}a&b \\ *&* \end{smallmatrix}\right)\in\GL(2,\fO/\pi^2\fO)} \int_{\fO} |a^2-b^2\delta+\pi^2x_1|^{s_1-1}\, \d x_1
\]
where $H_\delta$ denotes the stabilizer of $\left(\begin{smallmatrix}1&0 \\ 0&-\delta \end{smallmatrix}\right)$ in $\GL(2)$.
Furthermore, one can transform the integral $\sum_{\left(\begin{smallmatrix}a&b \\ *&* \end{smallmatrix}\right)\in\GL(2,\fO/\pi^2\fO)} \int_{\fO} |a^2-b^2\delta+\pi^2x_1|^{s_1-1}\, \d x_1$ to
\[
q^8(1-q^{-1})^2 \times \left\{ q^{-1} + q^{-2}\sum_{a\in\fO/\pi^2\fO }\int_\fO |a^2-\delta+\pi^2 x_1|^{s_1-1} \d x_1       \right\} .
\]
For $\delta\not\in(F^\times)^2$, it is easy to deduce the assertion from this.
Hence, we suppose $\delta=1$ from here.
When $v\not\in\Sigma_2$, the sum $q^{-1} + q^{-2}\sum_{a\in\fO/\pi^2\fO }\int_\fO |a^2-1+\pi^2 x_1|^{s_1-1} \d x_1$ equals
\[
q^{-1}+q^{-1}(q-2)+2\sum_{\alpha\in\pi\fO}q^{-s_1-1}\int_\fO |\alpha+\pi x_1|^{s_1-1} \d x_1=\frac{(1-q^{-1})(1+q^{-s_1})}{1-q^{-s_1}}.
\]
The remaining case $\delta\in(\fO^\times)^2$, $v\in\Sigma_2$ is easily proved, because it is sufficient to replace $a^2-b^2\delta$ by $2ab$ in the above integral.
\end{proof}


\section{Shintani double zeta functions}\label{double-zeta}

\subsection{Global zeta integral and zeta function}
Let $\d g$ denote a right Haar measure on $G(\A)$ normalized by
\[
\d g=\d b \, \d^\times c \, \d^\times a \quad \text{for} \quad g= (a , \begin{pmatrix}1&0\\b&c\end{pmatrix})\in G(\A).
\]
For $\Phi\in\cS(V(\A))$ and $\us =(s_1,s_2)\in\C^2$, the global zeta integral $Z(\Phi,\us)$ for Shintani's double zeta function is defined by
\begin{align}
Z(\Phi,\us):=\int_{G(\A)/G(F)} |\tau_1(g)|^{s_1}|\tau_2(g)|^{s_2} \sum_{x\in V^0(F)} \Phi(g \cdot x) \, \d g .
 \label{DefGlobalZeta}
\end{align} 
Let $S$ be a finite subset of $\Sigma$. For each class $\delta_S\in F_S^\times/(F_S^\times)^2$, we define
\[
V^0(F_S,\delta_S):= \{ x\in V^0(F_S) \mid P(x)\in \delta_S (F_S^\times)^2 \},
\]
where $V^0$ is the set of regular points in $V$ defined by \eqref{intro5}. For $\delta_S\in F_S^\times/(F_S^\times)^2$, $\omega_S\in \widehat{F_S^\times/(F_S^\times)^2}$, $\Phi\in\cS(V(F_S))$ and $\us=(s_1,s_2)\in\C^2$, set
\[
Z_S(\Phi,\us,\delta_S):= \int_{V^0(F_S,\delta_S)} |x_1|^{s_1-1}|P(x)|^{s_2-1}\Phi(x)\, \d x,
\]
\[
\tilde Z_S(\Phi,\us,\omega_S):= \int_{V^0(F_S)} |x_1|^{s_1-1}|P(x)|^{s_2-1} \omega_S(P(x))\, \Phi(x)\, \d x.
\]
The Fourier analysis on the finite group $F_S^\times/(F_S^\times)^2$ yields the relations
\begin{align}
&\tilde Z_S(\Phi,\us,\omega_S)=\sum_{\delta_S\in F_S^\times/(F_S^\times)^2} \omega_S(\delta_S)\times Z_S(\Phi,\us,\delta_S),
 \label{ZetaSOmega-Delta} \\
&Z_S(\Phi,\us,\delta_S)=|F_S^\times/(F_S^\times)^2|^{-1}\sum_{\omega_S\in \widehat{F_S^\times/(F_S^\times)^2}} \omega_S(\delta_S) \times \tilde Z_S(\Phi,\us,\omega_S). \label{ZetaSDelta-Omega}
\end{align}
For $x=\left(\begin{smallmatrix} x_1 & x_{12}  \\ x_{12} & x_2 \end{smallmatrix}\right) \in V(\A)$, the value $\Phi(x)$ will be denoted by $\Phi(x_1,x_{12},x_2)$. 
For simplicity we write $Z_\inf(\Phi,\us,\delta_\infty)$ for $Z_{\Sigma_\inf}(\Phi,\us,\delta_{\Sigma_\infty})$ and $\tilde Z_\inf(\Phi,\us,\omega_\infty)$ for $\tilde Z_{\Sigma_\inf}(\Phi,\us,\omega_{\Sigma_\infty})$. 

Set $K:=\prod_{v\in\Sigma} K_v$, where $K_v$ was defined in \eqref{eq:K}.
Then, $K$ is a maximal compact subgroup of $\GL(2,\A)$.
We choose a Haar measure $\d k$ such that $\int_K \d k=1$.
A test function $\Phi\in \cS(V(\A))$ is said to be $K$-spherical if $\Phi(k\cdot x)=\Phi(x)$ holds for any $k\in K$ and $x\in V(\A)$.

\begin{prop}\label{prop:globalana}
$Z(\Phi,\us)$ converges absolutely for $\Re(s_1)>1$ and $\Re(s_2)>1$ in the sense that
$$
\int_{G(\A)/G(F)} |\tau_1(g)|^{\Re(s_1)}|\tau_2(g)|^{\Re(s_2)} \sum_{x\in V^0(F)} |\Phi(g \cdot x)| \, \d g<+\infty.
$$
In addition, $Z(\Phi,\us)$ is holomorphic on the region $\Re(s_1)>1$, $\Re(s_2)>1$.
If $\Phi$ is $K$-spherical, then $Z(\Phi,\us)$ is meromorphically continued to $\C^2$.
\end{prop}
\begin{proof}
The absolute convergence can be proved by a direct calculation, cf. \cite[Lemma 3]{Shintani}.
The meromorphic continuation will be proved in Section \ref{sec:prin}.
\end{proof}

For any test function $\Phi\in\cS(V(\A))$, there exists a finite subset $S$ of $\Sigma$ such that 
\begin{equation}\label{eq:test}
S\supset\Sigma_\inf, \quad \Phi=\Phi_S\otimes\Phi^S_0, \quad \Phi_S\in\cS(V(F_S)), \quad \Phi_0^S=\otimes_{v\not\in S}\Phi_{v,0}
\end{equation}
where $\Phi_{v,0}$ denotes the characteristic function of $V(\fO_v)$. Fix such a finite set $S$ once and for all. Let $(F^{\times})^{S}=\{x\in \A^\times|\,x_v=1\,(v\in S)\}$ be the restricted direct product of the multiplicative groups $F_v^\times \,(v\not\in S)$. Let $t=t(F,S)$ be the $S$-class number of $F$, i.e., $t=\#(\A^\times/F^\times F_S^\times \prod_{v\not\in S}\fO_v^\times)$. Then there exists $t$ elements $\alpha_1,\dots,\alpha_t$ of $\A^{1}\cap (F^{\times})^{S}$ such that 
\begin{align}
\A^\times=\bigsqcup_{j=1}^t  (F_S^\times\prod_{v\notin S}\fO_v^\times) \alpha_j  F^\times. 
 \label{SclassDec}
\end{align}
Set $\gamma_{jk}:=(\alpha_j, \diag(1,\alpha_k))\in G(\A)$,
\[
\Gamma_{S,jk}:= G(F)\cap (\gamma_{jk}^{-1}G(\A,S)\gamma_{jk}), \quad L_{S,jk}:=V(F)\cap (\rho(\gamma_{jk})^{-1}V(\A,S)).
\]
where $G(\A,S)=G(F_S)\times \prod_{v\notin S}(G(F_v)\cap K_v)$ and $V(\A,S)=V(F_S)\times\prod_{v\notin S} V(\fO_v)$. Note that $\Gamma_{S,jk}$, when viewed as a subgroup of $G(F_S)$, is discrete in $G(F_S)$ acting on the discrete subset $L_{S,jk}$ of $V(F_S)$. For each $\delta_S\in F_S^\times/(F_S^\times)^2$, we set $L_{S,jk}(\delta_S):=L_{S,jk}\cap V^0(F_S,\delta_S)$ and
\[
\xi^S(\us,\delta_S):=  \sum_{1\leq j,k\leq t} \sum_{x\in \Gamma_{S,jk}\bsl L_{S,jk}(\delta_S)} \frac{1}{\#(\Gamma_{S,jk,x})}\frac{1}{|P_1(x)|_S^{s_1} |P_2(x)|_S^{s_2}}
\]
where $\Gamma_{S,jk,x}=\{ \gamma\in \Gamma_{S,jk}\mid \gamma\cdot x=x\}$. This is viewed as a natural generalization of the Shintani double zeta function \eqref{intro1} over an arbitrary number field $F$. Indeed, if $F=\Q$, then $t=1$ and $\xi^S(\us,\delta_S)$ reduces to the series \eqref{intro6}. 

\begin{lem} \label{lem:Globalzeta-Doublezeta}
Suppose $\Re(s_1)>1$ and $\Re(s_2)>1$. Then, 
\begin{equation}\label{eq:globallocal}
Z(\Phi,\us)=\frac{2^{|S|}}{|2|_S}\Delta_F^{-\frac{3}{2}}\times \prod_{v\in S\cap\Sigma_\fin} (1-q_v^{-1})^{-2} \times \sum_{\delta_S\in F_S^\times/(F_S^\times)^2} Z_S(\Phi_S,\us,\delta_S)\times \xi^S(\us,\delta_S).
\end{equation}
The series $\xi^S(\us,\delta_S)$ is absolutely convergent, defining a holomorphic function on the region $\Re(s_1)>1$, $\Re(s_2)>1$,
and $\xi^S(\us,\delta_S)$ is meromorphically continued to $\C^2$.
\end{lem}
\begin{proof} Although the proof is standard to experts, we include it for the sake of completeness. The absolute convergence stated in Proposition~\ref{prop:globalana} guarantees that we can freely exchange the order of integrals and summations. The decomposition \eqref{SclassDec} in conjunction with $\A=F+ F_S \prod_{v\not\in S}\fO_v$ yields the disjoint decomposition $G(\A)/G(F)=\bigcup_{1\leq\mu,\mu \leq t} \prod_{v\not\in S}G(\fO_v)\,(G(F_S)/\Gamma_{S,\mu\nu})\,\gamma_{\mu\nu}$. Applying this to the formula \eqref{DefGlobalZeta} and by \eqref{HaarMesDec}, we get
\begin{align*}
Z(\Phi,\us)\,\Delta_F^{3/2} &=\sum_{\mu,\nu}\int_{G(F_S)/\Gamma_{S,jk}}|\tau_1(g_S\gamma_{\mu\nu})|^{s_1}|\tau_2(g_S\gamma_{\mu\nu})|^{s_2}\sum_{x\in V^0(F)}\Phi(g_S \gamma_{\mu\nu} x)\,\d g_{S}.
\end{align*}
By \eqref{eq:test}, the non-vanishing of the prime-to-$S$ part of $\Phi(g_{S}\gamma_{\mu\nu}x)$ implies $x\in V(F)\cap \rho(\gamma_{\mu\nu}^{-1})V(\A,S)=L_{S,\mu\nu}$. We further decompose the set $V^0(F)$ to the disjoint union of $V^0(F)\cap V(F_S,\delta_{S})$ with $\delta_{S}\in F_S^\times/(F_S^\times)^2$ to extend the above computation as
\begin{align*}
& \sum_{\mu,\nu}\sum_{\delta_S} \int_{G(F_S)/\Gamma_{S,\mu,\nu}}|\tau_1(g_S\gamma_{\mu\nu})|^{s_1}|\tau_2(g_S\gamma_{\mu\nu})|^{s_2}\sum_{x\in L_{S,\mu\nu}(\delta_S)}\Phi_S(g_S x_S)\,\d g_{S}
\\
&=\sum_{\mu,\nu}\sum_{\delta_S} \int_{G(F_S)/\Gamma_{S,\mu\nu}}|\tau_1(g_S\gamma_{\mu\nu})|^{s_1}|\tau_2(g_S\gamma_{\mu\nu})|^{s_2}
\sum_{x\in \Gamma_{S,\mu\nu}\bsl L_{S,\mu\nu}(\delta_S)}\sum_{\gamma \in \Gamma_{S,\mu\nu}/\Gamma_{S,\mu\nu,x}} \Phi_S(g_S (\gamma x)_S)\,\d g_{S}.
\end{align*}
For $j=1,2$, we have
\begin{align}
|\tau_{j}(g_S \gamma_{\mu\nu})|=\frac{|P_{j}(g_S(\gamma x)_{S})|_S}{|P_j(x_S)|_S} \quad g_S\in G(F_S),\,\gamma \in \Gamma_{S,jk},x\in L_{S,\mu\nu}.
\label{Globalzeta-Doublezeta-1}
 \end{align}
Indeed, $|P_j(g_{S}\gamma_{\mu\nu}(\gamma x))|$ equals 
\begin{align*}
|P_j(g_{S}(\gamma x)_{S})||P_{j}(\gamma_{\mu\nu}(\gamma x)^{S})|
=|P_{j}(g_S(\gamma x)_S)|_S |\tau_{j}(\gamma_{\mu\nu}\gamma^S)||P_j(x^{S})|_S,
\end{align*}
where the subscript $S$ (resp. superscript $S$) means the $S$-component (resp. prime-to-$S$ component) of an adele. Since $\gamma_{\mu\nu}\gamma^{S}\gamma_{\mu\nu}^{-1}\in \prod_{v\not\in S}K_v$, $|\tau_{j}(\gamma_{\mu\nu}\gamma^{S})|$ equals $|\tau_{j}(\gamma_{jk})|$ which is $1$ due to the assumption $\alpha_{\mu},\alpha_{\nu}\in \A^1$. Since $x\in V(F)$, the product formula shows $|P_j(x^{S})|=|P_j(x_S)|_S^{-1}$. Thus $|P_j(g_{S}\gamma_{\mu\nu}(\gamma x))|=|P_{j}(g_S(\gamma x)_S)|_S|P_j(x_S)|_S^{-1}$ on one hand. On the other hand, $|P_j(g_{S}\gamma_{\mu\nu}(\gamma x))|=
|\tau_{j}(g_S \gamma_{\mu\nu})|$ because $|P_{j}(\gamma x)|=1$ due to $\gamma x\in V(F)$. 

Now applying \eqref{Globalzeta-Doublezeta-1} to the last formula of $Z(\Phi,\us)$, we proceed as follows
\begin{align*}
&Z(\Phi,\us)\Delta_F^{3/2} \\
&=\sum_{\mu,\nu}\sum_{\delta_S} \sum_{x\in \Gamma_{S,\mu\nu}\bsl L_{S,\mu\nu}(\delta_S)} \int_{G(F_S)/\Gamma_{S,\mu\nu}} \sum_{\gamma \in \Gamma_{S,\mu\nu}}
 |P_1(g_S(\gamma x)_S|_{S}^{s_1}|P_2(g_S(\gamma x)_S)|_S^{s_2} \Phi_S(g_S(\gamma x)_S)\d g_S \\
&\times \#(\Gamma_{S,\mu\nu,x})^{-1}|P_1(x_S)|_S^{-s_1}|P_2(x_S)|_S^{-s_2}.
\end{align*}
Note that the $\gamma$-summation and the $g_S$-integration are combined to yield an integral over the whole group $G(F_S)$ which is settled as follows. Since $x_{S}\in V(F_S,\delta)$, we see that $x_S$ and the matrix $\tilde \delta_S=\diag(1,-\delta_S)$ belong to the same $G(F_S)$-orbit. Then using the second equality of the first formula in Lemma~\ref{lem:coordinate-expression}, we see
\begin{align*}
\int_{G(F_S)}|P_1(g_S x_S)|_S^{s_1}|P_2(g_S x_S)|_S^{s_2}\Phi_S(g_S x_S)\d g_S
&=\int_{G(F_S)}|P_1(g_S \tilde \delta_S)|_S^{s_1}|P_2(g_S \tilde\delta_S)|_S^{s_2}\Phi_S(g_S \tilde \delta_S)\d g_S
\\
&=\{\prod_{v\in S\cap \Sigma_\fin}(1-q_v^{-1})^{-2}\}\,2^{|S|}|2|_S^{-1}Z(\Phi_S,\us,\delta_S).
\end{align*}
This completes the proof of the first assertion. The second assertion follows from \eqref{eq:globallocal} and Proposition \ref{prop:globalana}.  
\end{proof} 
For $\omega_S\in \widehat{F_S^\times/(F_S^\times)^2}$, define 
$\tilde \xi^S(\us,\omega_S)$ by
\begin{equation}\label{eq:elemchar}
\tilde \xi^S(\us,\omega_S):=\frac{2^{|S|+1}}{|F_S^\times/(F_S^\times)^2|\, |2|_S} \sum_{\delta_S \in F_S^\times/(F_S^\times)^2} \omega_S(\delta_S) \times \xi^S(\us,\delta_S).
\end{equation}
By the Fourier inversion on the finite group $F_S^\times/(F_S^\times)^2$, we can write, for $\delta_S\in F_S^\times/(F_S^\times)^2$,
\begin{equation}\label{eq:charelem}
\xi^S(\us,\delta_S)= \frac{|2|_S}{2^{|S|+1}}\sum_{\omega_S \in \widehat{F_S^\times/(F_S^\times)^2}} \omega_S(\delta_S) \times \tilde\xi^S(\us,\omega_S).
\end{equation}
From \eqref{eq:charelem} and \eqref{eq:globallocal}, and \eqref{ZetaSOmega-Delta}, we obtain  
\begin{equation}\label{eq:glchar}
Z(\Phi,\us)=\frac{1}{2} \Delta_F^{-\frac{3}{2}}\times \prod_{v\in S\cap\Sigma_\fin} (1-q_v^{-1})^{-2} \times \sum_{\omega_S\in \widehat{F_S^\times/(F_S^\times)^2}} \tilde Z_S(\Phi_S,\us,\omega_S)\times \tilde\xi^S(\us,\omega_S).
\end{equation}
For each $\us\in\C^2$ and each $\delta_S\in F_S^\times/(F_S^\times)^2$, it follows from Lemmas \ref{lem:nonv} and \ref{lem:Globalzeta-Doublezeta} that there exists a test function $\Phi_S\in \cS(V^0(F_S))$ $(\mathrm{Supp}\Phi_S\subset V^0(F_S,\delta_S))$ such that  $Z(\Phi_S,\us,\delta_S)\neq 0$ and
\begin{equation}\label{eq:zeta2019Ap9}
\xi^S(\us,\delta_S)=\frac{|2|_S}{2^{|S|}}\Delta_F^{\frac{3}{2}}\times \prod_{v\in S\cap\Sigma_\fin} (1-q_v^{-1})^{2} \times \frac{1}{Z(\Phi_S,\us,\delta_S)}  \times Z(\Phi,\us). 
\end{equation}
When $\Phi$ is $K$-spherical, Lemmas \ref{lem:nonv} is not available, and so the condition $Z(\Phi_S,\us,\delta_S)\neq 0$ is not ensured. In such a case, we should use Lemmas \ref{lem:snonv1}, \ref{lem:snonv2} and \ref{lem:padicnonva}.

\subsection{An integration formula} \label{sec:LL-formula} 
The following formula \eqref{LLformula} plays a key role in the proof of our main result.
It was given by Labesse and Langlands in the study of the stabilization of the trace formula for $\SL(2)$; see \cite[Chapter 3]{HW}.
Here we generalize the supports of test functions to the non-compactly supported case. 
The equality
\begin{align}
\int_{F^\times \bsl \A^\times}|c^2|^{\sigma}\sum_{z\in F^\times} \phi(c^2xz)\,\d^\times c
=\frac{1}{2} \sum_{\chi}\,\chi(x)\,\int_{\A^\times}\chi(c)|c|^{\sigma}\phi(c)\,\d^\times c,\quad x\in \A^1,
 \label{LLformula}
\end{align}
holds, where $\chi$ moves over all the characters of $F^\times (\A^\times)^2\bsl \A^\times$, and $\phi\in C(\A)$ and $\sigma\in \C$ are such that the function $|c|^{\sigma}\phi(c)$ on $\A^\times$ is integrable 
and 
\begin{align}
\sum_{\chi}\left|\int_{\A^\times}|c|^{\sigma}\chi(c)\phi(c)\,\d^\times c\right|<+\infty. 
 \label{LLformula-1}
\end{align}
\begin{proof}
Set $U=F^\times \bsl \A^1$ and let $U^2$ denote the image of the map $\iota: U \rightarrow U$ defined as $\iota(u)=u^2$ for $u\in U$, i.e., $U^2=\{u^2|\,u\in U\}$. Since $U$ is a compact abelian group, $U^2$ is a compact subgroup of $U$. Let $\frac{\d t}{t}$ be the Haar measure on $\R_{+}$ and $\d u$ the Haar measure on $U$ such that the product measure $\d u\otimes \tfrac{\d t}{t}$ on $U \times \R_{+}$ corresponds the Haar measure on $F^\times \bsl \A^\times=U \times \R_{+}$ prescribed in \S~\ref{sec:notations}. Since $U$ is compact and $\iota:U\rightarrow U$ is continuous
, there exists a Haar measure $\d x$ on $U^2$ such that $\int_{U^2}h(x)\d x=\int_{U}h(u^2)\,\d u$ for any continuous function $h$ on $U^2$. In particular, $\vol(U^2)=\vol(U)$, which implies that the quotient measure on $U^2\bsl U$ has the total volume $1$. Since $t\mapsto t^2$ is bijective on $\R_{+}$, we have a natural isomorphism $F^\times (\A^\times)^2\bsl \A^\times\cong U^2\bsl U$. Set $f(c)=\sum_{z\in F^\times}\phi(zc)$ for $c\in \A^\times$. Then $f(c)$ is a continuous function on $F^\times\bsl\A^\times$. Since the function $|c|^{\sigma}|\phi(c)|$ on $\A^\times$ is supposed to be integrable, we have that $|c|^{\sigma} f(c)$ is integrable on $F^\times\bsl \A^\times$. For any character $\chi$ of the compact group $U^2\bsl U\cong F^\times (\A^\times)^2\bsl \A^\times$,  
\begin{align*}
\int_{F^\times \bsl \A^\times}|c|^{\sigma}\chi(c)f(c)\d^\times c
&=\int_{\R_+}t^{\sigma} \tfrac{\d t}{t}\int_{U} \chi(u)f(tu)\,\d u \\
&=\int_{\R_{+}}t^{\sigma} \tfrac{\d t}{t} \int_{u\in U^2\bsl U} \chi(u) \int_{v\in U}f(utv^2)\,\d v\,\d u
\\
&=2\int_{\R_{+}}(t^2)^{\sigma}\tfrac{\d t}{t}  \int_{u\in U^2\bsl U}\chi(u) \int_{v\in U}f(u(tv)^2)\,\d v\,\d u
\\
&=2\int_{u\in U^2\bsl U}\chi(u)\,\int_{F^\times \bsl \A^\times}|c^2|^{\sigma}f(uc^2)\,\d^\times c\,\d u.
\end{align*}
This formula in conjunction with \eqref{LLformula-1} shows that the function $\tilde f(x) = 2\int_{F^\times \bsl \A^\times} |c^2|^{\sigma} f(c^2 x)\d^\times c$ on $U^2\bsl U\cong F^\times (\A^\times)^2\bsl \A^\times$ together with its Fourier transform is integrable. As noted above, the total measure of $U^2\bsl U$ is $1$. Thus by the Fourier inversion formula on $U^2\bsl U$, 
$$
\tilde f(x)=\sum_{\chi}\chi(x)\times 
\int_{F^\times \bsl \A^\times}|c|^{\sigma}\chi(c)f(c)\d^\times c. 
$$
This yields the desired formula.   
\end{proof}

\subsection{Explicit formula}\label{sec:explicit}

In this subsection, we prove the following theorem, which is our main result.
\begin{thm}[Explicit formula]\label{thm:global}
For each finite set $S\supset \Sigma_\inf$ and each $\omega_S\in \widehat{F_S^\times/(F_S^\times)^2}$, we have
\[
\tilde\xi^S(\us,\omega_S)= \zeta_F^S(s_1)  \, \zeta_F^S(2s_1+2s_2-1) \sum_\chi \frac{L^S(s_2,\chi)  }{ L^S(2s_1+s_2,\chi)\, N(\ff_{\chi}^S)^{s_1} }
\]
where $\chi=\otimes_v\chi_v$ moves over all quadratic characters satisfying $\otimes_{v\in S}\chi_v=\omega_S$ and we set $N(\ff_{\chi}^S):=\prod_{v\not\in S}\#(\fO_v/\ff_{\chi_v})$ with $\ff_{\chi_v}$ being the conductor of $\chi_v$. The above series is absolutely convergent for $\Re(s_1)>1$ and $\Re(s_2)>1$.
\end{thm}

For the proof of the last statement of Theorem~\ref{thm:global}, we need a lemma. 
\begin{lem}\label{lem:conductor-series} Let $S$ be a finite subset of $\Sigma$ containing $\Sigma_\infty$ and $\omega_S=\otimes_{v\in S}\omega_v$ a character of $F_S^\times/(F_S^\times)^2$. Then for any $\sigma>1$, we have $
\sum_{\chi}{N(\ff_{\chi}^S)^{-\sigma}}<+\infty$, where $\chi$ runs through all real valued idele class characters of $F^\times$ such that $\chi_v=\omega_v$ for $v\in S$. 
\end{lem}
\begin{proof} Let $\fO$ be the integer ring of $F$. For a given invertible ideal $\ff\subset \fO$, the number of $\chi$ such that $\ff_{\chi}=\ff$ is bounded by $O_{\epsilon}(N(\ff)^{\epsilon})$ for any $\epsilon>0$. To prove this, it suffices to estimate the order of the group $C(\ff)=F^\times (\A^\times)^2 U(\ff) \bsl \A^\times$, where $U(\ff)=\prod_{v\in S(\ff)}(1+\ff\fO_v)\times \prod_{v\in \Sigma_\fin-S(\ff)}\fO_v^\times$ with $S(\ff)=\{v\in \Sigma_\fin|\,\ff\fO_v\subset \pi_v\fO_v\}$. Since $(F_\infty^\times)^0\subset (\A^\times)^2$, $C(\ff)$ is a quotient group of $F^\times (F_{\infty}^\times)^0 U(\ff)\bsl \A^\times$ which is finite. Let $j_{\ff}$ be the quotient from $C(\ff)$ to the group $C(\fO)=F^\times (\A^\times)^2 U(\fO) \bsl \A^\times$. Then the kernel of $j_\ff$ is isomorphic to the group $((F^\times (\A^\times)^2)\cap U(\ff))\bsl U(\fO)$. Thus there exists a surjective map from $((\A^\times)^2\cap U(\ff))\bsl U(\fO)\cong \prod_{v\in S(\ff)}((\fO_v^\times)^2 (1+\ff\fO_v)\bsl \fO_v)$ onto ${\rm ker}(j_\ff)$. Therefore, 
\begin{align*}
\#(C(\ff))\leq \#(C(\fO)) \times \#{\rm ker}(j_\ff)
&\leq \#(C(\fO))\times \#\biggl(\prod_{v\in S(\ff)}(\fO_v^\times)^2(1+\ff\fO_v)\bsl \fO_v^\times\biggr) \\
&\leq \#(C(\fO))\times \#\biggl(\prod_{v\in S(\ff)}(\fO_v^\times)^2\bsl \fO_v^\times\biggr)
\end{align*}
If $v\in S(\ff)$ is non-dyadic, then $\#((\fO_v^\times)^2\bsl \fO_v^\times)=4$. Hence we have a constant $C_0>0$ such that $\#(C(\ff))\leq C_0\,4^{\# S(\ff)}$ for all $\ff$. The bound $4^{\#S(\ff)}\ll_{\epsilon} N(\ff)^{\epsilon}$ for any $\epsilon>0$ gives us the required majorant $N(\ff)^{\epsilon}$. Let $\chi$ be a real valued idele class character of $F^\times$ such that $\otimes_{v\in S}\chi_v=\omega_S$. Then $\ff_{\chi_v}=\ff_{\omega_v}$ is independent of $\chi$. Now choose $\epsilon>0$ such that $\sigma>1+\epsilon$; then we complete the proof by    
$$\sum_{\chi}{N(\ff_{\chi}^{S})^{-\sigma}} \ll_{\epsilon} 
\prod_{v\in S}\#(\fO/\ff_{\omega_v})^{\sigma}\times \sum_{\ff} {N(\ff)^{-(\sigma-\epsilon)}} \ll_{\sigma,\omega_S} \zeta_{F}^{S}(\sigma-\epsilon)<+\infty.
$$
\end{proof}

For any $\Re(s)>1$ and any real valued idele class character $\chi$ of $F^\times\bsl \A^\times$, by looking at the Dirichlet series expressions, we easily see that
$$
|L^{S}(s,\chi)|\leq \zeta_{F}^{S}(\Re(s)), \quad |L^{S}(s,\chi)|^{-1}\leq \zeta_{F}^{S}(\Re(s)). 
$$
Thus, for any $\epsilon>0$, there exists a constant $C(\epsilon)>0$ such that 
\begin{align}
\sum_{\chi} \left|\frac{L^S(s_2,\chi)  }{ L^S(2s_1+s_2,\chi)\, N(\ff_{\chi}^S)^{s_1} }\right|\leq C(\epsilon)\,\sum_{\chi}\frac{1}{N(\ff_{\chi}^{S})^{1+\epsilon}},\quad \Re(s_1)\geq 1+\epsilon,\,\Re(s_2)\geq 1+\epsilon. 
 \label{thm:global-f1}
\end{align}
The series of the majorant is convergent by Lemma~\ref{lem:conductor-series}. 

The first assertion of Theorem~\ref{thm:global} follows from the following lemma, Theorem~\ref{thm:local}, and \eqref{eq:glchar}. 

\begin{lem}\label{lem:LL}
Let $\Phi\in \cS(V(\A))$. For any $\us=(s_1,s_2)\in \C^2$ with $\Re(s_1)>1$ and $\Re(s_2)>1$, 
\[
Z(\Phi,\us)=\frac{1}{2}\sum_\chi \int_{\A^\times} \int_{ \A^\times}  \int_{ \A} |a|^{s_1+2s_2} |c|^{s_2} \chi(c) \Phi(a ,ab, ab^2-ac ) \, \d b \, \d^\times c \, \d^\times a .
\]
where $\chi$-summation taken over all characters of $F^\times (\A^1)^2 \bsl \A^1$, is absolutely convergent. If $\Phi$ is decomposed as in \eqref{eq:test} for some $S$, we have 
\[
Z(\Phi,\us)=\frac{1}{2}\Delta_F^{-\frac{3}{2}} \sum_{\chi = \otimes_v \chi_v} \tilde Z_S(\Phi_S,\us,\chi_S)\times \prod_{v\in S\cap\Sigma_\fin}(1-q_v^{-1})^{-2} \times   \prod_{v\not\in S} (1-q_v^{-1})^{-2} \tilde Z_v(\Phi_{v,0},\us,\chi_v)
\]
where $\chi_S=\otimes_{v\in S}\chi_v$.
\end{lem}
\begin{proof} A direct computation yields the $G(F)$-orbit decomposition $V^0(F)=\bigcup_{z\in F^\times/(F^\times)^2}G(F)\,\tilde z$, where $\tilde z=\diag(1,-z)$ for $z\in F^\times$, and that the stabilizer of $\tilde z$ in $G(F)$ coincides with $E=\{(1,\diag(1,\epsilon)|\,\epsilon\in \{\pm 1\}\}\cong \Z/2\Z$. If we write $g=\left(a,\left(\begin{smallmatrix} 1 & 0 \\ b & c \end{smallmatrix}\right)\right)\in G(\A)$, then the relation $x=\rho(g)\tilde z$ yields $x_1=a$, $x_{12}=ab$ and $x_2=a(b^2-zc^2)$. Thus by the absolute convergence in Proposition~\ref{prop:globalana}, we compute as follows  
{\allowdisplaybreaks\begin{align*}
Z(\Phi,\us)&=\sum_{z\in F^\times/(F^\times)^2} \int_{G(\A)/E}|\tau_1(g)|^{s_1}|\tau_2(g)|^{s_2}\Phi(g\,\tilde z)\,\d g\\
&=\sum_{z\in F^\times/(F^\times)^2} \int_{c\in \A^\times/\{\pm 1\}} \biggl(
\int_{\A^\times}\int_{\A}|a|^{s_1}|a^2c^2|^{s_2}\Phi(a,ab,a(b^2-zc^2)\,\d^\times a\,\d b \biggr) \,\d^\times c
\\
&=\sum_{z\in F^\times/(F^\times)^2} \int_{c\in \A^\times/F^\times }\,\sum_{\tau \in F^\times/\{\pm 1\}} \biggl(\int_{\A^\times}\int_{\A} |a|^{s_1}|a^2\tau^2c^2|^{s_2}\Phi(a,ab,a(b^2-z\tau^2c^2)\,\d^\times a\,\d b\biggr) \,\d^\times c \\
&=\int_{c\in F^\times\bsl \A^\times}  |c^2|^{s_2} \sum_{z\in F^\times} \biggl(\int_{\A^\times} \int_{\A} |a|^{s_1}|a^2|^{s_2}\Phi(  a , ab , ab^2-ac^2z ) \, \d b \, \d^\times a \biggr)\,\d^\times c
\\
&=\int_{c\in F^\times\bsl \A^\times}  |c^2|^{s_2} \sum_{z\in F^\times} \phi_{\us}(c^2z)\d^\times c,
\end{align*}}where $\phi_\us$ is a continuous function on $\A$ defined by 
$$
\phi_\us(c)=\int_{\A} \int_{\A^\times } |a|^{s_1+2s_2-1}\Phi(  a , b , a^{-1} b^2-ac ) \, \d b \, \d^\times a, \quad c\in \A.
$$
which is absolutely convergent for $\Re(s_1+2s_2-1)>1$ due to $\Phi\in \cS(\A)$. Let $\Phi_v^{+}\in \cS(F_v)$ be a non-negative majorant of $\Phi_v$ for $v\in S$. Then from 
$$\int_{\A^\times}|c|^{\Re(s_2)}|\phi_\us(c)|\,\d^\times c
 \leq \prod_{v\in \Sigma}\tilde Z_v(\Phi_v^{+},\Re(\us),{\bf 1})
$$
and Theorem~\ref{thm:local}, the integral of $|c|^{\Re(s_2)}|\phi_\us(c)|$ over $\A^\times$ is majorized by $\prod_{v\in S}\tilde Z_v(\Phi_v^{+},\us,{\bf 1})\times \zeta_F^{S}(\Re(s_1))\zeta_F^S(\Re(s_2))\zeta_F^{S}(2\Re(s_1)+2\Re(s_2)-1)\zeta_F^S(2\Re(s_1)+s_2)^{-1}$. From this, together with the convergence of local zeta integrals recalled in \S~\ref{sec:BasicFactsLocalZeta}), we see that $|c|^{\Re(s_2)}|\phi_\us(c)|$ is integrable on $\A^\times$ when $\Re(s_1)>1$, $\Re(s_2)>1$.  
From Theorem~\ref{thm:local}, 
\begin{align*}
\sum_{\chi}\left|\int_{\A^\times} |c|^{s_2}\phi_\us(c)\chi(c)\d^\times c\right|& \leq \prod_{v\in S}\tilde Z_v(\Phi_v^{+},\us,{\bf 1})\times |\zeta_F^{S}(s_1)\zeta_F^{S}(2s_1+2s_2-1)| \\
&\quad \times \sum_{\chi}\left|\frac{L^S(s_2,\chi)}{L^{S}(2s_1+s_2,\chi)\,N(\ff_\chi^{S})^{s_2}}\right|,
\end{align*}
Since the $\chi$-summation is convergent as we already saw above, the condition \eqref{LLformula-1} for $\phi_\us$ is also satisfied for $\us$ with $\Re(s_1)>1$ and $\Re(s_2)>1$. 
Now we apply \eqref{LLformula} to the last formula of $Z(\Phi,\us)$ to complete the proof.  
\end{proof}

\subsection{Principal part of global zeta integrals and meromorphic continuation}\label{sec:prin}

In this section, we will study the meromorphic continuation of the global zeta integral $Z(\Phi,\us)$ using Shintani's arguments in \cite{Shintani}.
For the argument, we need not only the usual global Fourier transform 
\[
\widehat\Phi(y):=\int_{V(\A)} \Phi(x) \, \psi_F(\langle x,y\rangle)\, \d x,
\]
but also partial Fourier transforms of $\Phi\in\cS(V(\A))$ such as
\[
\Phi^{(3)}(x_1,x_2,y_3):=\int_\A \Phi(x_1,x_2,x_3) \, \psi_F(x_3y_3) \, \d x_3.
\]
In the same way, we have the partial Fourier transform $\Phi^{(1)}$ and $\Phi^{(2)}$ with respect to the first variable and the second variable, respectively. We set $\Phi^{(i,j)}=(\Phi^{(i)})^{(j)}$ for $1\leq i\not=j\leq 3$. The truncated global zeta integral 
\[
Z_+(\Phi,\us):=\int_{G(\A)/G(F), \, |\tau_2(g)|>1} |\tau_1(g)|^{s_1}|\tau_2(g)|^{s_2} \sum_{x\in V^0(F)} \Phi(g \cdot x) \, \d g 
\]
is absolutely convergent for $\Re(s_1)>1$ defining a holomorphic function on the domain $\Re(s_1)>1$ of $\C^2$. Set 
\[
T(\Phi,s):=\int_{\A^\times} \int_\A |a|^{s-1} \Phi(a,b,a^{-1}b^2) \, \d b \, \d^\times a, \quad \Phi\in \cS(V(\A)), 
\]
which is seen to be absolutely convergent on $\Re(s)>1$ as follows. For a finite set of places $S\subset \Sigma$, we define its local counterpart $T_{S}(\Phi_S,s)$ for $\Phi_S\in \cS(V(F_S))$; $T_S(\Phi_S,s)$ is easily seen to be absolutely convergent for $\Re(s)>1/2$ and $T_S(\Phi_S,s)$ is meromorphically continued to $\C$.
If $v\in \Sigma_\fin$, then a computation shows
\begin{equation}\label{eq:texp12}
T_v(\Phi_{v,0},s)=\frac{1+q_v^{-s}}{1-q_v^{-2s+1}}=\frac{1-q_v^{-2s}}{(1-q_v^{-s})(1-q_v^{-2s+1})}. 
\end{equation}
Hence, for any $\Phi$ satisfying \eqref{eq:test}, we have the demanded absolute convergence on $\Re(s)>1$ as well as the formula 
\begin{equation}\label{eq:singular3.17}
T(\Phi,s)=\Delta_F^{-1/2}\frac {\zeta_F^S(s)\zeta_F^S(2s-1)}{\zeta_F^S(2s)}\, T_S(\Phi_S,s),
\end{equation}
which establishes a meromorphic continuation of $T(\Phi,s)$ to $\C$.
\begin{lem}\label{lem:holT}
Set $u_v(s):=\Gamma(2s-1)$ if $v\in\Sigma_\C$, $u_v(s):=\Gamma(\frac{2s-1}{2})$ if $v\in\Sigma_\R$, and $u_v(s):=1-q_v^{-2s+1}$ if $v\in\Sigma_\fin$.
Then, $\prod_{v\in S}u_v(s)\times T_S(\Phi_S,s)$ is holomorphic on the domain $\Re(s)>0$.
Hence, $T(\Phi,s)$ is holomorphic on the domain $\Re(s)>0$ except for $s=1$ by \eqref{eq:singular3.17}.
\end{lem}
\begin{proof}
It is sufficient to prove the assertion for each place $v\in S$.
For $v\in\Sigma_\inf$, the assertion follows from the general theory, see \cite[Chapter 5]{Igusa}.
Let $v\in\Sigma_\fin$ and consider the integral $T_v(\Phi_v,s)=\int_{F_v^\times}\int_{F_v} |a|_v^{s_1-1} \Phi_v(a,b,a^{-1}b^2) \, \d b \, \d^\times a$.
Since $\Phi_v$ is locally constant, we may suppose that $\Phi$ is the characteristic function of $\{ (x_1,x_2,x_3) \mid x_j\in \pi^{l_j}(\alpha_j+\pi^{m_j}\fO_v)     \;\; (1\leq j\leq 3)\}$ for some $l_j\in\Z$, $\beta\in\fO_v$ and large numbers $m_j\in\Z_{>0}$ without loss of generality.
Then, one finds that $T_v(\Phi_v,s)$ is holomorphic unless $\alpha_1=\alpha_2=\alpha_3=0$ .
Therefore, the assertion follows from \eqref{eq:texp12}.
\end{proof}

\begin{lem}\label{lem:princ}
If $\Re(s_1)>1$ and $\Re(s_2)>1$, then
\begin{align*}
Z(\Phi,\us)=& Z_+(\Phi,\underline{s})+Z_+(\widehat\Phi,s_1,\frac{3}{2}-s_1-s_2) + \frac{c_F}{2s_1+2s_2-3} T(\widehat\Phi,s_1) - \frac{c_F}{2s_2} T(\Phi,s_1) \\
& +\frac{c_F}{2s_2-2} \zeta(\widehat\Phi^{(3)}(0,0,\cdot),s_1) - \frac{c_F}{2s_1+2s_2-1} \zeta(\Phi^{(3)}(0,0,\cdot),s_1),
\end{align*}
where we set $c_F:=\mathrm{Res}_{s=1}\zeta_F(s)=\vol(F^\times\bsl \A^1)$.
\end{lem}
\begin{proof}
The assertion can be proved by the same argument as in the proof of \cite[Lemma 4]{Shintani}.
\end{proof}

\begin{cor}\label{cor:ac1}
For any $\delta_S\in F^\times_S/(F^\times_S)^2$, it follows that
\[
(s_2-1)(2s_1+2s_2-3)\xi^S(\us,\delta_S)
\]
are analytically continued to a holomorphic function on the domain
\[
\mathscr{D}_1:=\{ (s_1,s_2)\in\C^2 \mid \Re(s_1)>1 \}.
\]
\end{cor}
\begin{proof}
Choose a test function $\Phi$ as \eqref{eq:test} and $\Phi_S=\otimes_{v\in S}\Phi_v\in\cS(V(F_S))$ as in Lemma \ref{lem:nonv}.
Then, one gets $T(\Phi,s_1)=0$ by $P(a,b,a^{-1}b^2)=0$.
We also get $\Phi^{(3)}(0,0,a)=0$ for any $a\in\A$ by the definition. 
Hence, $(s_2-1)(2s_1+2s_2-3)Z(\Phi,\us)$ is analytically continued to a holomorphic function on $\mathscr{D}_1$ by Lemma \ref{lem:princ}, and so the assertion follows from \eqref{eq:zeta2019Ap9}.
\end{proof}

We set, for $j=1$ or 2, 
\[
V^1(F):=\{ x\in V(F) \mid P(x)\neq 0 ,\quad P(x)\not\in  (F^\times)^2   \}, \quad V^2(F):= V^0(F) \setminus V^1(F) ,
\]
\[
Z_j(\Phi,\us):=\int_{G(\A)/G(F)} |\tau_1(g)|^{s_1}|\tau_2(g)|^{s_2} \sum_{x\in V^j(F)} \Phi(g \cdot x) \, \d g ,
\]
\[
Z_{j,+}(\Phi,\us):=\int_{G(\A)/G(F), \, |\tau_2(g)|>1} |\tau_1(g)|^{s_1}|\tau_2(g)|^{s_2} \sum_{x\in V^j(F)} \Phi(g \cdot x) \, \d g.
\]
Clearly,
\[
Z(\Phi,\us)=Z_1(\Phi,\us)+Z_2(\Phi,\us)\;\; \text{and} \;\; Z_+(\Phi,\us)=Z_{1,+}(\Phi,\us)+Z_{2,+}(\Phi,\us).
\]
For $h=k\begin{pmatrix}a&0 \\ 0&c\end{pmatrix} \begin{pmatrix} 1&0 \\ b&1 \end{pmatrix}\in \GL(2,\A)$ and $k\in K$, let $t(h):=\left| \frac{a}{c}\right|^{1/2}$. 
For $s\in\C$, define the Eisenstein series 
\[
E(h,s):= \sum_{\gamma\in \GL(2,F)/B(F)} t(h\gamma)^s,\quad 
\]
where $B$ denotes the Borel subgroup consisting of lower triangular matrices in $\GL(2)$.
It is known that $E(h,s)$ absolutely converges and locally uniformly for $\Re(s)>2$ and $h\in\GL(2,\A)$, and satisfies the functional equation
\[
\widehat\zeta_F(2-s)\, E(h,2-s) = \widehat\zeta_F(s)\, E(h,s).
\]
Furthermore, $E(h,s)$ is holomorphic for $\Re(s)>1$ except for a simple pole at $s=2$.
By \cite[Lemma 6.1 (iii)]{W}, $\rho_0=\mathrm{Res}_{s=1}E(h,2s)=\dfrac {\mathrm{Res}_{s=1} \widehat\zeta_F(s)}{2\widehat\zeta_F(2)}=\dfrac {c_F \pi^{r_1}(2\pi)^{r_2}}{2\Delta_F^{\frac 12}\zeta_F(2)}$.

The group $G''=\GL(1)\times \PGL(2)$ acts on $V$ as $g\cdot x=\frac{a}{\det(h)} hx{}^t\!h$, $g=(a,h)\in G''$, $x\in V$.
For $\Phi\in\cS(V(\A))$ and $s\in\Bbb C$, we set
\[
\mathfrak Z(\Phi,s):=\int_{G''(\A)/G''(F)} |a|^{2s} \sum_{x\in V^1(F)} \Phi(g\cdot x) \, \d g.  
\]
The zeta integral $\mathfrak Z(\Phi,s)$ absolutely converges for $\Re(s)>3/2$.
Its meromorphic continuation to $\C$ was proved in \cite{Yukie}, and its explicit formula was studied in \cite{Datskovsky} and also in \cite{HW}.
Yukie \cite{Yukie} showed that $\mathfrak Z(\Phi,s)$ has a simple pole at $s=\frac 32$ and at most a double pole at $s=1$.

\begin{prop}\label{prop:eisen}
Suppose $\Phi$ is $K$-spherical.
Then, one has
\[
Z_1(\Phi,\us)=2\int_{G''(\A)/G''(F)} |a|^{s_1+2s_2} \, E(h,2s_1) \sum_{x\in V^1(F)}\Phi(g\cdot x)\, \d g.
\]
Therefore, $Z_1(\Phi,\us)$ is meromorphically continued to the domain
\[
\mathscr{D}_2:=\{ (s_1,s_2)\in\C^2 \mid \Re(s_1+2s_2)>2+\max\{\Re(s_1),1-\Re(s_1),1\} \}.
\]
In particular, $(s_1-1) \, Z_1(\Phi,\us)$ is a holomorphic function on $\mathscr{D}_2$.
Furthermore, if $\Re(s_2)>1$, then one has
\[
\mathrm{Res}_{s_1=1}Z_1(\Phi,\us) =  2\rho_0\, \mathfrak Z(\Phi,s_2+\frac{1}{2}).
\] 
\end{prop}
\begin{proof}
The equality follows from direct calculation.
The meromorphic continuation can be proved by using the above mentioned basic properties and a Fourier expansion of $E(h,s)$; see \cite[Lemma 6.1]{W}.
\end{proof}
\begin{prop}\label{prop:eisen+}
Suppose $\Phi$ is $K$-spherical.
By the same argument as in Proposition \ref{prop:eisen}, one has
\[
Z_{1,+}(\Phi,\us)=2\int_{G''(\A)/G''(F), |a|>1} |a|^{s_1+2s_2} \, E(h,2s_1) \sum_{x\in V^1(F)}\Phi(g\cdot x)\, \d g.
\]
From this one finds that $Z_{1,+}(\Phi,\us)$ is meromorphically continued to $\C^2$.
In addition,
\[
(s_1-1) \,  Z_{1,+}(\Phi,\us)
\]
is a holomorphic function on $\C^2$.
\end{prop}
\begin{proof}
By choosing a Siegel set, any element $g$ in $G''(\A)/G''(F)$, which satisfies $|a|>1$, is expressed as
\[
g=\left(a,k \begin{pmatrix} b&0 \\ 0&1 \end{pmatrix}  \begin{pmatrix} 1&0 \\ c&1 \end{pmatrix} \right) ,\quad a,b\in\A^\times, \;\; c\in\A/F , \;\; k\in K , \;\; |a|>1 , \;\;   |b|>1/2.
\]  
For $x=\begin{pmatrix} x_1&x_{12} \\ x_{12}&x_2 \end{pmatrix}\in V^1(F)$, one has
\[
g\cdot x= \left(1,k\begin{pmatrix} 1&0 \\ b^{-1}c&1 \end{pmatrix}\right) \cdot \begin{pmatrix} ab x_1& ax_{12} \\ ax_{12}&  ab^{-1}x_2 \end{pmatrix}.
\]
Therefore, from $x_1\neq 0$, one finds that $\Phi(g\cdot x)$ is rapidly decreasing for the directions $|a|\to \inf$ and $|b|\to \inf$.
Hence, this case is also proved by using the above mentioned basic properties and a Fourier expansion of $E(h,s)$; see \cite[Lemma 6.1]{W}.
\end{proof}

\begin{prop}\label{prop:splitpart}
Suppose that a finite subset $S$ of $\Sigma$ contains $\Sigma_\inf\sqcup\Sigma_2$.
For each test function $\Phi$ satisfying \eqref{eq:test}, one has
\[
Z_2(\Phi,\us)=2^{|S|}\Delta_F^{-\frac{3}{2}} \prod_{v\in S\cap\Sigma_\fin}(1-q_v^{-1})^{-2} \times Z_S(\Phi_S,\us,1) \times \frac{\zeta_F^S(s_1)^2 \, \zeta_F^S(2s_2) \, \zeta_F^S(2s_1+2s_2-1)}{\zeta_F^S(s_1+2s_2) \zeta_F^S(2s_1)}
\]
for $\Re(s_1)>1$ and $\Re(s_2)>1$.
Hence, $Z_2(\Phi,\us)$ is meromorphically continued to $\C^2$.
\end{prop}
\begin{remark}
It follows from Lemmas \ref{lem:Rfcteq} and \ref{lem:polelocalp} that
\[
\sin\left(\frac{\pi s_1}{2}\right)^{\#\Sigma_\R} \,\left\{ \prod_{v\in S\cap\Sigma_\fin}(1-q_v^{-2s_1})^\cs \right\} \, (s_1-1)^2 (2s_2-1)\zeta_F^S(2s_1)\, Z_2(\Phi,\us)
\]
is holomorphic on $\mathscr{D}_2$, where $\cs$ is the constant given in Lemma \ref{lem:polelocalp}.
\end{remark}
\begin{proof}
It follows from direct calculation that
$$
Z_2(\Phi,\us)=2^{|\Sigma_\inf|}\Delta_F^{-\frac{3}{2}}  Z_S(\Phi_S,\us,1) \prod_{v\in S\cap\Sigma_\fin}2(1-q_v^{-1})^{-2}  
 \prod_{v\not\in S} \left\{2(1-q_v^{-1})^{-2}Z_v(\Phi_{v,0},\us,1) \right\}.
$$
Note that the factor $2$ comes from the fact that, in the change of variable for $\begin{pmatrix} a &ab \\ ab & ab^2-ac^2\end{pmatrix}$, the domain of integration for $c^2$ doubles. The $v$-factors outside $S$ are computed by Corollary~\ref{cor:local}. 
\end{proof}

For each $\delta_S\in F_S^\times/(F_S^\times)^2$ and each positive real number $X$, we set
\[
Z_S(\Phi_S,\us,\delta_S,X):=\int_{V(F_S,\delta_S), \, |P(x)|>X} |x_1|^{s_1-1}|P(x)|^{s_2-1}\Phi_S(x) \, \d x.
\]
It is obvious that $Z_S(\Phi_S,\us,\delta_S,X)$ is absolutely convergent and holomorphic for $\Re(s_1)>0$.
For $l,r\in\N$ and $\delta\in F^\times_v/(F_v^\times)^2$, we set
\[
Y_v(s,l,\delta):= \begin{cases} q_v^{(-2s+1)r}+(1-q_v^{-s})\sum_{j=0}^{r-1}q_v^{(-2s+1)j} & \text{if $\chi_\delta$ is unramified and $l=2r$,} \\  \sum_{j=0}^{r} q_v^{(-2s+1)j} & \text{if $\chi_\delta$ is ramified and $l=2r+1$,}  \\  0 & \text{otherwise} \end{cases}
\]
where $\chi_\delta(\;\;):=(\delta, \;\; )_v$.
Note that $Y_v(s,0,\delta)=1$ if $\delta\in\fO^\times$.
\begin{lem}\label{lem:local2019.3.26}
Let $v\in\Sigma_\fin\setminus \Sigma_2$, $\delta\in F^\times_v$ $(|\delta|_v=1$ or $q_v^{-1})$, and $\tilde\Phi_{v,r}$ denote the characteristic function of $\{x\in V(\fO_v) \mid |P(x)|_v\in \pi^r\fO_v^\times\}$ for $r\in\N$.
Then,
\[
Z_v(\tilde\Phi_{v,r},\us,\delta)= \frac{1}{2}\times \frac{L_v(s_1,\trep_v) \, L_v(s_1,\chi_\delta)}{L_v(2s_1,\trep_v)} \times  Y_v(s_1,r,\delta)  \times \frac{1}{q_v^{rs_2}} \times \frac{1}{N(\ff_{\chi_\delta})^{s_2}}.
\]
\end{lem}
\begin{proof}
This can be proved by Lemma \ref{lem:coordinate-expression} and an argument similar to the proof of Theorem \ref{thm:local}.
\end{proof}

\begin{prop}\label{prop:splitpart+}
Suppose that a finite subset $S$ of $\Sigma$ contains $\Sigma_\inf\sqcup\Sigma_2$.
Choose a test function $\Phi$ as in \eqref{eq:test}.
Then, one has
\[
Z_{2,+}(\Phi,\us)=2^{|S|}\Delta_F^{-\frac{3}{2}}\prod_{v\in S\cap\Sigma_\fin}(1-q_v^{-1})^{-2} \times\frac{\zeta_F^S(s_1)^2}{\zeta_F^S(2s_1)} \sum_\fa \frac{Z_S(\Phi_S,\us,1,N(\fa)^2)}{N(\fa)^{2s_2}} \, \prod_{v\not\in S}Y_v(s_1,2r_{\fa,v},1).
\]
where $\fa$ moves over all integral ideals prime to $S$, $N(\fa)$ denotes the norm of $\fa$, and $r_{\fa,v}$ denotes the power of the prime ideal corresponding to $v$ in $\fa$.
Furthermore, $Z_{2,+}(\Phi,\us)$ absolutely converges for $\Re(s_1)>1$, and is meromorphically continued to $\C^2$.
In addition, $Z_{2,+}(\Phi,\us)$ is holomorphic on the domain $\Re(s_1)\geq 1/2$ except for $s_1=1$.
\end{prop}
\begin{proof}
The equality follows from Lemma \ref{lem:local2019.3.26} and the same argument as in the proof of Proposition \ref{prop:splitpart}.
The other assertions follow from the facts that $\prod_{v\not\in S}Y_v(s_1,2r_{\fa,v},1)$ is bounded by $(N(\fa)+N(\fa)^{-2\Re(s_1)+1})^{N(\fa)^{1/2}}$ and we have $|Z_{2,+}(\Phi,\us)| <  Z_{2,+}(|\Phi|,(\Re(s_1),\Re(s_2)+M)$ for any $M>0$.
\end{proof}


\begin{cor}\label{s_1=1}
Suppose that $\Phi\in\cS(V(\A))$ is $K$-spherical and satisfies \eqref{eq:test}.
Then, 
\[
(s_1-1)^2 s_2(s_2-1)(2s_1+2s_2-1)(2s_1+2s_2-3) Z(\Phi,\us)
\]
is analytically continued to a holomorphic function on the domain $\Re(s_1) \geq 1/2$.
In addition, $Z(\Phi,\us)$ is meromorphically continued to $\C^2$, $Z_1(\Phi,\us)$ has a simple pole at $s_1=1$, and $Z_2(\Phi,\us)$ has a double pole at $s_1=1$;
\[
\lim_{s_1\to1} \frac{\partial}{\partial s_1}(s_1-1)^2 Z(\Phi,\us) =  2\rho_0\, \mathfrak Z(\Phi,s_2+\frac{1}{2}) + \lim_{s_1\to1} \frac{\partial}{\partial s_1}(s_1-1)^2Z_2(\Phi,\us)
\]
holds for any $s_2\in\C$.
\end{cor}
\begin{proof}
The meromorphic continuation follows from Lemmas \ref{lem:holT} and \ref{lem:princ} and Propositions \ref{prop:eisen+} and \ref{prop:splitpart+}.
The equality follows from Propositions \ref{prop:eisen} and \ref{prop:splitpart}.
\end{proof}

\begin{lem}\label{cor:ac2}
For any $\delta_S\in F^\times_S/(F^\times_S)^2$, 
\[
\Gamma\left(\frac{s_1+1}{2}\right)^{-\#\Sigma_\R}\sin\left(\frac{\pi s_1}{2}\right)^{\#\Sigma_\R} \times  \prod_{v\in (S\cap\Sigma_\fin)\cup\Sigma_2}(1-q_v^{-2s_1})  \times (s_1-1)^2 (2s_2-1) \, \zeta_F^S(2s_1)\, \xi^S(\us,\delta_S)
\]
is a holomorphic function on $\mathscr{D}_2$.
\end{lem}
\begin{proof}
For each $\delta_S\in F_S^\times/(F_S^\times)^2$ and $\us\in\C^2$, by Lemmas \ref{lem:snonv1}, \ref{lem:snonv2}, and \ref{lem:padicnonva}, there exists a compactly supported $\prod_{v\in S}K_v$-spherical function $\Psi_S\in\cS(V(F_S))$ such that
\[
Z(\Psi_S,\us,\delta_S) \times \Gamma\left(\frac{s_1+1}{2}\right)^{\#\Sigma_\R}\times \prod_{v\in\Sigma_\fin\cap S} (1+q_v^{-s_1})^{-1}\neq 0
\]
and $Z(\Psi_S,\us,u_S)=0$ for any $u_S\neq \delta_S$.
Furthermore, choosing $\Phi=\Psi_S\otimes_{v\not\in S}\Phi_{v,0}\in\cS(V(\A))$, we derive from Propositions \ref{prop:eisen} and \ref{prop:splitpart} that
\[
\sin\left(\frac{\pi s_1}{2}\right)^{\#\Sigma_\R} \times  \prod_{v\in (S\cap\Sigma_\fin)\cup\Sigma_2}(1-q_v^{-s_1})  \times (s_1-1)^2 (2s_2-1) \, \zeta_F^S(2s_1)\, Z(\Phi,\us)
\]
is holomorphic on $\mathscr{D}_2$.
Therefore, the assertion follows from \eqref{eq:zeta2019Ap9}. 
\end{proof}

\begin{cor}\label{cor:ac3}
For any $\delta_S\in F^\times_S/(F^\times_S)^2$, the function
\begin{multline*}
\Gamma\left(\frac{s_1+1}{2}\right)^{-\#\Sigma_\R}\sin\left(\frac{\pi s_1}{2}\right)^{\#\Sigma_\R} \times  \prod_{v\in (S\cap\Sigma_\fin)\cup\Sigma_2}(1-q_v^{-2s_1})  \\
\times (s_1-1)^2 (2s_2-1) \, \zeta_F^S(2s_1)\,  (s_2-1)(2s_1+2s_2-3)\, \xi^S(\us,\delta_S)
\end{multline*}
is a holomorphic function on $\C^2$.
\end{cor}
\begin{proof}
This follows from Corollary \ref{cor:ac1} and Lemma \ref{cor:ac2}.
By \cite[Theorem 2.5.10]{H}, it has an analytic continuation to all of $\mathbb{C}^2$, since the convex hull of $\mathscr{D}_1$ and $\mathscr{D}_2$ is $\C^2$.
\end{proof}

For $\alpha$, $\beta\in\C$, we set
\begin{equation}\label{eq:utab}
\ut(\alpha,\beta):=\left(\alpha-\frac{\beta}{2}+\frac{1}{2}, \; \frac{\beta}{2}\right) \in \C^2.
\end{equation}
\begin{prop}\label{prop:hol}
Assume that $F$ is a totally real field, that is, $\Sigma_\C=\emptyset$.
Let $l$, $m\in\Z$, $l\leq 0$, $m\geq 4$ and $m$ is even.
For any $\delta_S\in F^\times_S/(F^\times_S)^2$, the function $\xi^S(\us,\delta_S)$ is holomorphic at $\us=\ut(l,m)\in(\frac{1}{2}\Z)^2$.
\end{prop}
\begin{proof} This follows from Corollary \ref{cor:ac3}.
Notice that, if $\Sigma_\C\neq \emptyset$, then one has $\zeta^S_F(2s_1)=\zeta^S_F(2l-m+1)=0$ by the functional equation.
Hence, in such a case, it might have a pole.
This is the reason that we assumed $\Sigma_\C=\emptyset$.
\end{proof}
\begin{remark}
Let us consider the case $m=2$ in Proposition \ref{prop:hol}.
In this case, the Dirichlet series agrees with the Shintani zeta function for the space of binary quadratic forms.
But the argument in Proposition \ref{prop:hol} cannot be applied to the case $m=2$ because of the pole at $s_2=1$.
However, one can study their special values at $l\in\Z_{\leq 0}$ by the same argument as in \cite{Shintani} if $F$ is a totally real field over $\Q$ and $\omega_v=\sgn$ for every $v\in\Sigma_\R$.
\end{remark}

\subsection{Functional equations}\label{fun}

Suppose that $S$ contains $\Sigma_\inf$.
Set
\[
\Xi^S(\us,\omega_S):=\frac{\zeta^S_F(2s_1)}{\zeta^S_F(s_1)}\, \tilde \xi^S(\us,\omega_S).
\]
\begin{thm}\label{thm:funct1}
Recall $\Gamma_S(s,\omega_S)$ from \eqref{eq:functquad} for $\omega_S\in 
\widehat{F_S^\times/(F_S^\times)^2}$. Then
\[
\Xi^S(s_1+s_2-\frac{1}{2},1-s_2,\omega_S)=\Gamma_S(s_2,\omega_S) \, \Xi^S(\us,\omega_S).
\]
\end{thm}
\begin{proof}
This follows from \eqref{eq:functquad} and Theorem \ref{thm:global}.
\end{proof}

Recall $\tilde G_v(\us,\chi_v,\omega_v)$ from \eqref{GG}, and define
\[
\tilde G_S(\us,\chi_S,\omega_S):=\prod_{v\in S}\tilde G_v(\us,\chi_v,\omega_v) .
\]
\begin{thm}\label{thm:funct2}
If $S$ contains $\Sigma_\inf\cup\Sigma_2\cup\{v\in\Sigma_\fin \mid \fd_v\neq 0\}$, then
\[
\Xi^S\left((s_1,\tfrac{3}{2}-s_1-s_2),\omega_S\right)=\Delta_F^{-\frac{3}{2}}\sum_{\chi_S\in\widehat{F_S^\times/(F_S^\times)^2}} \tilde G_S(\us,\chi_S,\omega_S) \, \Xi^S(\us,\chi_S).
\]
\end{thm}
\begin{proof}
It follows from Lemma \ref{lem:princ} that
\begin{equation*}
Z(\hat\Phi,\us)=Z\left(\Phi,(s_1,\tfrac{3}{2}-s_1-s_2)\right) \qquad (\forall \Phi \in\cS(V(\A))). 
\end{equation*}
Hence, one can derive this functional equation by using Lemmas \ref{lem:nonv} and \ref{lem:localfechar} and \eqref{eq:glchar}.
For $v\in\Sigma_\fin$, we note that $\Phi_{v,0}= \hat \Phi_{v,0}$ if and only if $v\not\in\Sigma_2\cup\{v\in\Sigma_\fin \mid \fd_v\neq 0\}$.
\end{proof}

By the functional equations of Theorems \ref{thm:funct1} and \ref{thm:funct2}, we conjecture that $\tilde\xi^S(\us,\omega_S)$ possesses a group of functional equations isomorphic to $D_{12}$.

\begin{remark}\label{rem:explicit}
In \cite[p.291]{IS2}, it is proved 
\[
\xi_j(\us)= \frac{\zeta(s_1)\zeta(2s_2) \zeta(2s_1+2s_2-1)}{2\zeta(2s_1)}\sum_{(-1)^{j-1} D>0} \frac{L(s_1,\chi_D)}{L(s_1+2s_2,\chi_D) \, |D|^{s_2}}.
\]
where $D$ moves over $1$ and fundamental discriminants.
An adelic version of this formula was given by Taniguchi in \cite[Proposition B.9]{Taniguchi}. 
Further, from their explicit formula and the functional equation of $L(s_1,\chi)$, one obtains
\[
\frac{\zeta(2-2s_1)}{\zeta(1-s_1)}\xi_j(1-s_1,s_1+s_2-\frac{1}{2})=2^{-s_1+1}\pi^{-s_1}\Gamma(s_1)\, \frac{\zeta(2s_1)}{\zeta(s_1)}\xi_j(\us) \begin{cases} \cos(\pi s_1/2) & \text{if $j=1$}, \\ \sin(\pi s_1/2) & \text{if $j=2$} \end{cases}
\]
which was proved in \cite[Theorem 1]{Shintani}.

Similarly, a functional equation of $\xi^S(\us,\delta_S)$ as $(1-s_1,s_1+s_2-\frac{1}{2})\leftrightarrow (s_1,s_2)$ also follows from Theorems \ref{thm:funct1} and \ref{thm:funct2}.
\end{remark}

\subsection{The zeta functions $D_m(s,\omega_S)$}\label{D_m}

Choose $m\in\Bbb Z$, $m\geq 1$ and recall the notation $\ut( \, \cdot \, ,m)$ defined in \eqref{eq:utab}.
For $s\in\C$, we set
\begin{align}\label{D_m}
D_m(s,\omega_S)&:=\frac{\zeta_F^S(2s-m+1)}{\zeta_F^S(s-\frac{m}{2}+\frac{1}{2})}  \, \tilde \xi^S(\ut(s,m),\omega_S) \\
&=\zeta_F^S(2s-m+1)\zeta_F^S(2s)\sum_\chi \frac{L^S(m/2,\chi)}{ L^S(2s-\frac{m}{2}+1,\chi)\, N(\ff_{\chi}^S)^{s-\frac{m}{2}+\frac{1}{2}} },\nonumber
\end{align}
where $\chi=\otimes_v\chi_v$ moves over all real valued characters satisfying $\otimes_{v\in S}\chi_v=\omega_S$.
\begin{cor}\label{cor:applicationtotarceformula}
Assume that $F$ is a totally real field, that is, $\Sigma_\C=\emptyset$.
Let $l$, $m\in\Z$, $l\leq 0$, $m\geq 4$ and $m$ is even.
The Dirichlet series $D_m(s,\omega_S)$ is holomorphic at $s=l$.
\end{cor}
\begin{proof}
This is a corollary of Proposition \ref{prop:hol}; see \eqref{eq:elemchar}.
\end{proof}

This corollary will be used in the study of equidistribution theorems of holomorphic Siegel cusp forms of general degree in \cite{KWY}, in particular, in the estimation of the unipotent contributions.

\begin{prop}\label{prop:simplepole3.17}
Let $m\in\Z$, $m\geq 1$, and $m\neq 2$.
The Dirichlet series $D_m(s,\omega_S)$ is holomorphic in the domain $\{s\in\C\mid \Re(s) \geq m/2\}$ except for $s=(m+1)/2$.
Further, it has the simple pole at $s=\frac{m+1}{2}$ if $m\neq 1$, and it has the double pole at $s=1$ if $m=1$. 
\end{prop}
\begin{proof}
For each $\delta_S\in F_S^\times/(F_S^\times)^2$ and $\us\in\C^2$, we choose a compactly supported $\prod_{v\in S}K_v$-spherical function $\Psi_S\in\cS(V(F_S))$ as in Lemmas \ref{lem:snonv1}, \ref{lem:snonv2}, and \ref{lem:padicnonva}, and set $\Phi=\Psi_S\otimes_{v\not\in S}\Phi_{v,0}\in\cS(V(\A))$.
Then, it follows from this test function $\Phi$, \eqref{eq:zeta2019Ap9} and Corollary \ref{s_1=1} that $\xi^S((s-\frac{m}{2}+\frac{1}{2},\frac{m}{2}),\delta_S)$ is holomorphic on $\Re(s)\geq m/2$ except for $s=(m+1)/2$.
Hence, the first assertion is proved.

For the case $m>2$, $\tilde \xi^S(\ut(s,m),\omega_S)$ always has a double pole at $s=\frac {m+1}2$ for any $\omega_S$, because $\xi^S(\ut(s,m),\delta_S)$ has a double pole at $s=\frac {m+1}2$ if and only if $\delta_S=1$.
Note that the double pole comes from $Z_2(\Phi,\us)$; $Z_1(\Phi,\us)$ has at most a simple pole at $s_1=1$; see Propositions \ref{prop:eisen} and \ref{prop:splitpart}. 
But it cancels with the simple pole of the denominator $\zeta_F^S(s-\frac m2+\frac 12)$.

Let us consider the case $m=1$, and choose a test function $\Phi=\Psi_S\otimes_{v\not\in S}\Phi_{v,0}$ as above.
By Lemma \ref{lem:princ}, $Z(\Phi,\ut(s,1/2))$ has a triple pole at $s=1$ if and only if $T(\widehat{\Phi},s)$ has a double pole at $s=1$.
It follows from \eqref{eq:singular3.17} that this is equivalent to that $T_S(\widehat{\Psi}_{S_1},1)$ does not vanish, where $S_1=S\cup\Sigma_2$.
Since $\widehat{\Psi}_{S_1}$ is now $\prod_{v\in S_1}K_v$-spherical, one has
\begin{align*}
T_S(\widehat{\Psi_{S_1}},1)= & \int_{F_{S_1}^\times}\int_{F_{S_1}} |a|\, \widehat{\Psi}_{S_1}((a,\begin{pmatrix} 1&0 \\ b&1 \end{pmatrix})\cdot \begin{pmatrix}1&0 \\ 0&0 \end{pmatrix}) \, \d b \, \d^\times a \\
= & \mathrm{(constant)} \times \int_{G''(F_S)} |a| \, \phi(h) \, \widehat{\Psi}_{S_1}( g\cdot \begin{pmatrix}1&0 \\ 0&0 \end{pmatrix}) \, \d g \\
= & \mathrm{(constant)} \times \int_{F_S^\times}\int_{F_S}\int_{F_S^\times} |a| \, |c|^{-1} \phi(\begin{pmatrix}1&b \\ 0& c\end{pmatrix}) \, \widehat{\Psi}_{S_1}( (a,\begin{pmatrix}1&b \\ 0& c\end{pmatrix})\cdot \begin{pmatrix}1&0 \\ 0&0 \end{pmatrix}) \, \d^\times c \, \d b \, \d^\times a \\
= & \mathrm{(constant)} \times \int_{F_S}  \widehat{\Psi}_{S_1}(a,0,0) \d a = \mathrm{(constant)} \times \int_{F_S}\int_{F_S}  \Psi_{S_1}(x_1,x_{12},0)  \d x_1 \, \d x_{12} 
\end{align*}
where $\phi$ is a left $\prod_{v\in S}K_v$-spherical function in $C_c^\inf(\PGL(2,F_S))$.
Therefore, choosing a test function $\Phi$, one can prove that $\xi^S(\ut(s,m),\delta_S)$ has a triple pole at $s=1$ if and only if $\delta_S=1$.
Thus, the assertion for $m=1$ follows from this fact and the simple pole of the denominator $\zeta_F^S(s)$ at $s=1$.
\end{proof}

\begin{remark}
It is also possible to prove the existence of the double pole of $D_1(s,\omega_S)$ at $s=1$ by using Corollary \ref{s_1=1} and a property of $\mathfrak Z(\Phi,s)$.
Actually, $\mathfrak Z(\Phi,s_2+\frac{1}{2})$ has a double pole at $s_2=1/2$ if and only if
$\int_{F_S}\int_{F_S}  \Psi_{S_1}(x_1,x_{12},0)  \d x_1 \, \d x_{12}\neq 0$.
(See \cite[Definitions (2.3), (2.11) and (2.14), Proposition (2.12), and Theorem (4.2)]{Yukie}.)
Furthermore, the double pole of $D_1(s,\omega_S)$ comes from it by Corollary \ref{s_1=1}. 
\end{remark}

\begin{remark}
As for $D_2(s,\omega_S)$ $(m=2)$, its poles were already studied by \cite{Yukie} and \cite{Datskovsky}. See also \cite[Chapter 4]{HW}.
The Dirichlet series $D_2(s,\omega_S)$ is holomorphic in the domain $\{s\in\C\mid \Re(s) \geq 3/4\}$ except for $s=1$, $3/2$.
Further, it has the simple pole at $s=\frac{3}{2}$, and it might have a simple or double pole at $s=1$.
In particular, it has a double pole at $s=1$ if and only if $\omega_S=\trep_S$.
\end{remark}

\begin{remark}
We can consider the Dirichlet series $D_\beta(s,\omega_S)$ for a complex number $\beta$ instead of an integer $m$. 
As in \cite[Theorem (2)]{GH}, it follows from Corollary \ref{s_1=1} that $D_\beta(s,\omega_S)$ is holomorphic in the domain $\{s\in\C\mid \Re(s)>\Re(\beta)/2\}$ except for $s=(\beta+1)/2$ if $\Re(\beta)\geq 1$.
However, it seems difficult to determine the order of the pole for $1\leq \Re(\beta)\leq 2$, $\beta\neq 1,2$, because Propositions \ref{prop:eisen} and \ref{prop:splitpart} are not available.
\end{remark}

\section{Application to central values of Dirichlet $L$-functions}\label{appl}

In this section, let $F=\Bbb Q$.
For each fundamental discriminant $D$, we write $\chi_D$ for the Dirichlet character corresponding to $\Q(\sqrt{D})$.
For $D=1$, $\chi_D$ means the trivial character.
Let $S_0$ be a finite set of prime numbers, and we choose $S$ as $S=S_0\sqcup\{\inf\}$.
Since $\chi_D$ is identified with a character on $\Q^\times \R_{>0}\bsl \A$ in the usual manner, we use the same notation $\chi_D$ for it, and we have the decomposition $\chi_D=\prod_{v\in\Sigma}\chi_{D,v}$.
Let $\fD$ denote the set of $1$ and all fundamental discriminants, and set
\[
\fD(\omega_S):=\{D\in\fD \mid \chi_{D,v}=\omega_v \quad (\forall v\in S)\}
\]
for $\omega_S=\prod_{v\in S}\omega_v\in \widehat{\Q_S^\times/(\Q_S^\times)^2}$.
We consider the following four cases
\begin{itemize}
\item[(i)] $\omega_\inf=\trep_\inf$ and $m\not\equiv 2 \mod 4$.
\item[(ii)] $\omega_\inf=\trep_\inf$ and $m\equiv 2 \mod 4$.
\item[(iii)] $\omega_\inf=\sgn$ and $m\not\equiv 0 \mod 4$.
\item[(iv)] $\omega_\inf=\sgn$ and $m\equiv 0 \mod 4$.
\end{itemize}
for $m\in\Z_{\geq 1}$, $\omega_\inf=\trep_\inf$ or $\sgn$. 
For each $m\in\Z_{\geq 1}$, a Dirichlet series $L_m(s,\omega_S)$ is defined by
\begin{multline*}
L_m(s,\omega_S):= D_m(s,\omega_S) \times 2(2\pi)^{-\frac{m}{2}} \Gamma(m/2)\\
\times \prod_{p\in S_0}\frac{1-\omega_p(p) p^{-2s+\frac{m}{2}-1}}{(1-p^{-2s})(1-p^{-2s+m-1})(1-\omega_p(p)p^{-\frac{m}{2}})} \times \begin{cases} \cos(m\pi/4) & \text{if (i),}  \\ \frac{\pi}{2} \sin (m\pi/4)  & \text{if (ii),}  \\   \sin(m\pi/4)  & \text{if (iii),}  \\ -\frac{\pi}{2}\cos(m\pi/4) & \text{if (iv).} \end{cases}
\end{multline*}
Note that $D>0$ if and only if $D\in \fD(\trep_\inf)$.

Following \cite{IS1}, we give a generalization of the Cohen's function \cite{C}.
Write
\begin{equation}\label{eq:divisorsetc}
\frac 1{L(s,\chi_D)}=\sum_{n=1}^\infty \mu(n)\chi_D(n) n^{-s},\quad \zeta(2s)\zeta(2s-m+1)=\sum_{n=1}^\infty \sigma_{m-1}(n) n^{-2s}.
\end{equation}
We also put $\delta_\omega:=\begin{cases} 0 & \text{if $\omega_\inf=\trep_\inf$,} \\ 1 & \text{if $\omega_\inf=\sgn$,} \end{cases}$
\[
\mathfrak{N}(\omega_S):= \{  N\in\N^* \mid \exists D\in\fD(\omega_S), \;\; \exists f\in\N^*, \;\; N= (-1)^{\delta_\omega} Df^2  \}.
\]
We set, for $N=(-1)^{\delta_\omega}Df^2\in\mathfrak{N}(\omega_S)$,
\begin{equation}\label{cohen}
H(m/2,N,\omega_S ):=\sum_{u|f} \mu(u)\chi_D(u) u^{\frac m2-1}\sigma_{m-1}(f/u) \times \begin{cases} L(1-\frac m2,\chi_D) & \text{if (i) or (iii),}  \\ L'(1-\frac m2,\chi_D)  & \text{if (ii) or (iv).} \end{cases}
\end{equation}
where $L'(s,\chi_D)$ is the derivative of $L(s,\chi_D)$.
When $\delta_\omega=\frac {1+(-1)^{\frac m2}}2$, $m$ even, and $S_0=\emptyset$, $H(m/2,N,\omega_\inf)$ is exactly the Cohen's function. (This is also $A_{m}(1/2,k)$ in \cite[p. 197]{GH}.)

\begin{prop} For $m\in\Bbb Z_{\geq 1}$,
\[
L_m(s,\omega_S)=\sum_{N\in \mathfrak{N}(\omega_S)} \frac{H(m/2,N,\omega_S)}{N^s}.
\]
This series is absolutely convergent for $\Re(s)>\frac {m+1}2$.
\end{prop}
\begin{proof} From the definition of $D_m(s,\omega_S)$ from (\ref{D_m}), 
the equality is proved by the functional equations of Dirichlet $L$-functions with real valued characters; see \eqref{eq:functquad} and \eqref{eq:gammaS}.
The range of the absolute convergence is proved by Theorem \ref{thm:global} for $m\geq 2$.
Let us consider the case $m=1$.
If $\Re(s)>1$, then the two series in \eqref{eq:divisorsetc} are absolutely convergent, and so one can reduce the absolute convergence of the series $L_m(s,\omega_S)$ for $Re(s)>1$ to that of the series $\sum_{D\in\fD(\omega_S)} |L(1/2,\chi_D)|\, |D|^{-s}$ for $Re(s)>1$, which was proved by \cite[Theorem (2)]{GH}.
\end{proof}

Consider $m=1$. Then $H(1/2,N,\omega_S)=L(1/2,\chi_D)\sum_{u|f} \mu(u)\chi_D(u) u^{-\frac 12}\sigma_{0}(f/u)$ for $N=(-1)^{\delta_\omega} Df^2$.
In particular, if $N$ is a fundamental discriminant, $H(1/2,N,\omega_S)=L(1/2,\chi_{N})$.
The average of $L(1/2,\chi_D)$ has been studied by many people (cf. \cite{GH}), and in particular, by \cite[Theorem (1)]{GH},
$$\sum_{|D|<x} L(1/2,\chi_D)=C x\log x+D x+O(x^{19/32+\epsilon}),
$$
for some constants $C,D$. The error term has been improved to $O(x^{\frac 12+\epsilon})$ \cite{Young}.

It may be interesting to obtain the average of $H(1/2,N,\omega_S)$, which is a weighted average of $L(1/2,\chi_D)$. 
If we know that $H(1/2,N,\omega_S)$ is non-negative, we can use Sato and Shintani's generalization of Landau's theorem \cite[Theorem 3]{SS}, and prove an asymptotic formula 
$$\sum_{N\leq x,\, N\in \mathfrak{N}(\omega_S)} H(1/2,N,\omega_S)= Ax\log x+Bx+O(x^{1/3}),
$$
for some constants $A,B$. However, it is not proved yet that $H(1/2,N,\omega_S)$ is non-negative. We prove

\begin{thm}\label{thm:asym} Let $S_0$ be a finite set of prime numbers such that $2\in S_0$, and let $S=S_0\sqcup\{\inf\}$. Let $\omega_S$ be a real valued character on $\Q_S^\times$. 
Then, for any $\epsilon>0$, there exist constants $A,B$ such that 
$$\sum_{N\leq x,\, N\in \mathfrak{N}(\omega_S)} H(1/2,N,\omega_S)= Ax\log x+Bx+O(x^{19/32+\epsilon}).
$$
\end{thm}
\begin{proof}
By the above proposition, $D_1(s,\omega_S)$ is absolutely convergent for $\Re(s)>1$.

Since $\Xi^S((s,1/2),\omega_S)=D_1(s,\omega_S)$, one has
\begin{multline}\label{eq:funct2019March252}
D_1(1-s,\omega_S)=\sum_{\chi_{1,S},\,\chi_{2,S},\,\chi_{3,S}\in \widehat{\Q_S^\times/(\Q_S^\times)^2} } \tilde G_S((1-s,s),\chi_{1,S},\omega_S)\, \Gamma_S(1-s,\chi_{1,S}) \\
\times \tilde G_S((1/2,s),\chi_{2,S},\chi_{1,S})\, \Gamma_S(1-s,\chi_{2,S})\, \tilde G_S((s,1/2),\chi_{3,S},\chi_{2,S}) \, D_1(s,\chi_{3,S})
\end{multline}
by Theorems \ref{thm:funct1} and \ref{thm:funct2}.

By the definition of $\Gamma_S(s,\chi_S)$ from \eqref{eq:functquad} and $\tilde G_v(\us,\chi_v,\omega_v)$ from \eqref{GG}, 
we have
\begin{eqnarray*} 
&&\Gamma_S(1-s,\chi_S)=2 (2\pi)^{s-1}\Gamma(1-s)\cos (1-s-\delta_\chi)\pi/2 \prod_{p\in S_0\atop \text{$\chi_p$ ramified}} p^{f_{\chi_p}(\frac 12-s)}\prod_{p\in S_0\atop \text{$\chi_p$ unramified}} \frac {1-\chi_p(p)p^{-s}}{1-\chi_p(p)p^{s-1}},\\
&& \tilde G((1-s,s),\chi_{1,S},\omega_S)= c_{\chi_1,\omega}\pi^{\frac 12-s}\frac {\Gamma((s+\delta_{\chi_1})/2)}{\Gamma((1-s+\delta_{\chi_1})/2)} \prod_{p\in S_0\atop \text{$\chi_{1p}$ ramified}} p^{f_{\chi_{1p}} s}\prod_{p\in S_0\atop \text{$\chi_{1p}$ unramified}} \frac {1-\chi_{1p}(p)p^{s-1}}{1-\chi_{1p}(p)p^{-s}},\\
&& \tilde G((1/2,s),\chi_{2,S},\chi_{1,S})=c_{\chi_2,\chi_1} \pi^{1-2s} \frac {\Gamma((s+\delta_{\chi_1}/2)}{\Gamma((1-s+\delta_{\chi_1})/2)}\frac {\Gamma((s+\delta_{\chi_2})/2)}{\Gamma((1-s+\delta_{\chi_2})/2)} \\
&&\phantom{xxxx}   \times \prod_{p\in S_0\atop \text{$\chi_{1p}$ ramified}} p^{f_{\chi_{1p}} s}\prod_{p\in S_0\atop \text{$\chi_{1p}$ unramified}} \frac {1-\chi_{1p}(p)p^{s-1}}{1-\chi_{1p}(p)p^{-s}}
\prod_{p\in S_0\atop \text{$\chi_{2p}$ ramified}} p^{f_{\chi_{2p}} s}\prod_{p\in S_0\atop \text{$\chi_{2p}$ unramified}} \frac {1-\chi_{2p}(p)p^{s-1}}{1-\chi_{2p}(p)p^{-s}},\\
&& \tilde G((s,1/2),\chi_{3,S},\chi_{2,S})=c_{\chi_3,\chi_2} \pi^{\frac 12-s}\frac {\Gamma((s+\delta_{\chi_2})/2)}{\Gamma((1-s+\delta_{\chi_2})/2)}
\prod_{p\in S_0\atop \text{$\chi_{2p}$ ramified}} p^{f_{\chi_{2p}} s}\prod_{p\in S_0\atop \text{$\chi_{2p}$ unramified}} \frac {1-\chi_{2p}(p)p^{s-1}}{1-\chi_{2p}(p)p^{-s}},
\end{eqnarray*}
where $c_{\chi_1,\omega}, c_{\chi_2,\chi_1}, c_{\chi_3,\chi_2}$ are constants, and
$f_{\chi_p}$ is the conductor of $\chi_p$.

There are 4 cases: (i) $\delta_{\chi_1}=\delta_{\chi_2}=0$; (ii) $\delta_{\chi_1}=0$, $\delta_{\chi_2}=1$; (iii) $\delta_{\chi_1}=1$, $\delta_{\chi_2}=0$; (iv) $\delta_{\chi_1}=\delta_{\chi_2}=1$.

We use the identities: $\Gamma(1-s)\Gamma(s)=\frac {\pi}{\sin(\pi s)}$ and 
$\Gamma(\frac s2)\Gamma(\frac {s+1}2)=2^{1-s}\sqrt{\pi}\Gamma(s)$. Then we can see that
\begin{eqnarray*} 
&& \tilde G_S((1-s,s),\chi_{1,S},\omega_S)\, \Gamma_S(1-s,\chi_{1,S})
 \tilde G_S((1/2,s),\chi_{2,S},\chi_{1,S})\, \Gamma_S(1-s,\chi_{2,S})\, \tilde G_S((s,1/2),\chi_{3,S},\chi_{2,S}) \\
&& =c_{\omega,\chi_1,\chi_2,\chi_3}
\prod_{p\in S_0\atop \text{$\chi_{1p}$ ramified}} p^{f_{\chi_{1p}} s}\prod_{p\in S_0\atop \text{$\chi_{1p}$ unramified}} \frac {1-\chi_{1p}(p)p^{s-1}}{1-\chi_{1p}(p)p^{-s}}
\prod_{p\in S_0\atop \text{$\chi_{2p}$ ramified}} p^{f_{\chi_{2p}} s}\prod_{p\in S_0\atop \text{$\chi_{2p}$ unramified}} \frac {1-\chi_{2p}(p)p^{s-1}}{1-\chi_{2p}(p)p^{-s}}\\
&& \phantom{xxxxxxxxxxxx} \times \begin{cases}  4(2\pi)^{-2s}\Gamma(s)^2 (\sin(\pi s/2))^2, &\text{if (i)}\\
                  2\pi \Gamma(s)^2 \sin(\pi s), &\text{if (ii) or (iii)} \\
                 4(2\pi)^{-2s}\Gamma(s)^2 (\cos(\pi s/2))^2, &\text{if (iv)}\end{cases}.
\end{eqnarray*}
where $c_{\omega,\chi_1,\chi_2,\chi_3}=c_{\chi_1,\omega}c_{\chi_2,\chi_1}c_{\chi_3,\chi_2}\prod_{p\in S_0\atop \text{$\chi_{1p}$ ramified}} p^{f_{\chi_{1p}}/2}\prod_{p\in S_0\atop \text{$\chi_{1p}$ ramified}} p^{f_{\chi_{1p}}/2}$.

By Proposition \ref{prop:simplepole3.17}, $D_1(s,\omega_S)$ is holomorphic for $\Re(s)\geq \frac 12$, except at $s=1$, where it has a double pole. The above functional equation provides the meromorphic continuation of $D_1(s,\omega_S)$ to the whole complex plane, and $D_1(s,\omega_S)$ is holomorphic for $\Re(s)\leq \frac 12$, except at $s=0$, where it may have a double pole.

Now, the Eisenstein series $(s-1)E(h,2s)$ is an entire function of finite order (\cite[Theorem 0.2]{Mu}). Hence from Proposition \ref{prop:eisen+},
$s^2(s-1)^2 Z_{1,+}(\Phi, (s,1/2))$ is an entire function of finite order.  
By Proposition \ref{prop:splitpart+}, $s^2(s-1)^2 Z_{2,+}(\Phi, (s,1/2))$ is an entire function of finite order.
Therefore, $s^2(s-1)^2 Z_+(\Phi, (s,1/2))$ is an entire function of finite order. 
Then by Lemma \ref{lem:princ}, $s^2(s-1)^2 Z(\Phi, (s,1/2))$ is an entire function of finite order, and hence $s^2(s-1)^2 D_1(s,\omega_S)$ is an entire function of finite order.

Since $|\Gamma(s)|\sim e^{-\frac {\pi}2 |t|} |t|^{\sigma-\frac 12} \sqrt{2\pi}(1+O(t^{-1}))$ for $s=\sigma+it$ and $|t|\to\infty$,
for $\epsilon>0$, $D_1(-\epsilon+it,\omega_S)\ll |t|^{1+2\epsilon}$. Therefore, we have, for $s=\sigma+it$, $0<\sigma<1$, by Phragmen-Lindelof principle,
\begin{equation*}\label{convexity}
D_1(\sigma+it,\omega_S)\ll |t|^{1-\sigma+2\epsilon}.
\end{equation*}

Now we note the relationship
$$L_1(s,\omega_S)=\text{(constant)}\times \prod_{p\in S_0} \frac {1-\omega_p(p) p^{-2s-\frac 12}}{(1-p^{-2s})^2} D_1(s,\omega_S).
$$
Hence $L_1(s,\omega_S)$ is holomorphic for $\Re(s)>0$, except at $s=1$, where it has a double pole. Also we have
$$L_1(\sigma+it,\omega_S)\ll |t|^{1-\sigma+2\epsilon}.$$

Now by Burgess' estimate (cf. \cite[p. 204]{GH}), $L(\frac 12,\chi_D)\ll D^{\frac 3{16}+\epsilon}$. 
Also by \cite[p. 205]{GH}, if $N=(-1)^{\delta_\omega}Df^2$, $|H(\frac 12,N,\omega_S)|\leq |L(\frac 12,\chi_D)|d(f) 2^{w(f)}$, where $w(f)$ is the number of distinct prime divisors of $f$ and $d(f)$ is the number of divisors of $f$. Since $2^{w(f)}, d(f)\ll f^\epsilon$ for any $\epsilon>0$,
$|H(\frac 12,N,\omega_S)|\ll N^{\frac 3{16}+\epsilon}$.
Also for $\sigma>1$,
$$B(\sigma)=\sum_{N\in \frak N(\omega_S)} |H(\frac 12,N,\omega_S)|N^{-\sigma}\ll \sum_{N=\pm Df^2} |L(\frac 12,\chi_D)| d(f)2^{w(f)}|D|^{-\sigma}f^{-2\sigma}.
$$
Here 
$$\sum_{f=1}^\infty d(f)2^{w(f)}f^{-2\sigma}\ll \sum_{f=1}^\infty d(f)f^{-2\sigma+\epsilon}=\zeta(2\sigma-\epsilon)^2=O(1).
$$
Now, by \cite[Theorem 2]{J}, 
$$\sum_{\pm D\leq x} L(\frac 12,\chi_D)^2=C x (\log x)^3+O(x(\log x)^{\frac 52+\epsilon}),
$$
for some constant $C>0$.
Hence by Cauchy-Schwarz inequality, 
$$\sum_{\pm D\leq x} |L(\frac 12,\chi_D)|\ll x (\log x)^{\frac 32}.
$$
Therefore, for $\sigma>1$,
$$
B(\sigma)\ll \sum_{\pm D} |L(\frac 12,\chi_D)| |D|^{-\sigma}\ll \int_1^\infty (\log x)^{\frac 32}x^{-\sigma}\, dx\ll (\sigma-1)^{-\frac 52}.
$$

Now we apply the truncated Perron's formula \cite[p. 70]{Titch}: For $c>1$,
$$\sum_{N\leq x, N\in\frak N(\omega_S)} H(\frac 12,N,\omega_S)=\frac 1{2\pi i}\int_{c-iT}^{c+iT} L_1(s,\omega_S)\frac {x^s}s\, ds
+O\left(x^{\frac {19}{16}+\epsilon}T^{-1}\right)+O(x^c B(c)T^{-1})+O(x^{\frac 3{16}+\epsilon}).
$$
Take $c=1+\frac 1{\log x}$. Then $B(c)\ll (\log x)^{\frac 52}$. We move the contour to $\Re(s)=\sigma_0$, $0<\sigma_0<1$. The double pole of $L_1(s,\omega_S)$ at $s=1$ gives rise to $Ax\log x+Bx$ for some $A,B$. Also
$$\left|\frac 1{2\pi i}\int_{\sigma_0-iT}^{\sigma_0+iT} L_1(s,\omega_S)\frac {x^s}s\, ds\right|\ll x^{\sigma_0} \int_{-T}^T |L_1(\sigma_0+it,\omega_S)||\sigma_0+it|^{-1}\, dt\ll x^{\sigma_0}T^{1-\sigma_0+2\epsilon}.
$$
The integrals on the line from $c+iT$ to $\sigma_0+iT$, and from $c-iT$ to $\sigma_0-iT$ give rise to $O(xT^{-1} (\log (xT^{-1})^{-1})$.
By taking $T=x^{(\frac {19}{16}-\sigma_0)/(2-\sigma_0)}$ and
$\sigma_0=\epsilon$, we have

$$\sum_{N\leq x, N\in\frak N(\omega_S)} H(\frac 12,N,\omega_S)=Ax\log x+Bx+O(x^{\frac {19}{32}+\epsilon}).
$$
This proves our theorem.
\end{proof}

\begin{remark} Let $L_1(s,\omega_S)=\sum_{N=1}^\infty a_{N,S} N^{-s}$. Let 
$$A(x)=\sum_{N<x} a_{N,S},\quad A_1(x)=\int_0^x A(t)\, dt=\sum_{N<x} a_{N,S}(x-N).
$$
Then we can show that $A_1(x)=x^2(A\log x+B)+O(x^{1+\epsilon}).$
\end{remark}

\begin{remark} For $m> 2$, we can show that $D_m(s,\omega_S)$ satisfies the functional equation: $D_m(\frac {m+1}2-s,\omega_S)$ is a linear combination of 
$\Gamma(s)\Gamma(s-\frac m2+\frac 12) (\text{$\sin(\pi s)$ or $\cos(\pi s)$}) D_m(s,\chi_S)$ with the gamma factors of finite places in $S$.
By the functional equation, $D_m(s,\omega_S)$ is holomorphic on the domain $\Re(s)\leq \frac{1}{2}$, $\frac m2\leq \Re(s)$ except for $s=0$, $\frac{m+1}{2}$.
Hence, since one has $\ut(s,m)\in\mathscr{D}_2$ if $m> 2$ and $\frac {m+1}4 \leq \Re(s) \leq \frac m2$, it follows from Lemmas \ref{lem:snonv2} and \ref{lem:padicnonva}, \eqref{eq:zeta2019Ap9}, Propositions \ref{prop:eisen} and \ref{prop:splitpart} that the poles of $\zeta_F^S(s-\frac{m}{2}+\frac{1}{2})\, D_m(s,\omega_S)$ are in $\{ 0, \frac{1}{2}, \dots,\frac{m+1}{2}\} \sqcup \{s\in\C \mid \Re(s)=1$ or $\frac{m-1}{2} \}$. 
When $S=\{\inf\}$ and $m>3$ is even, their poles were studied in \cite[Proposition 3.6]{IS2} more precisely.

As for $m=2$, the functional equation of $D_2(s,\omega_S)$ is different from the others, because $\tilde G_\inf(\us,*,\trep)$ has a pole at $(s_1,s_2)=(1-s,s-\frac{1}{2})$ (cf. \eqref{eq:funct2019March25}). We refer to \cite{Datskovsky} for its functional equation. 
\end{remark}


\begin{thebibliography}{99}

\bibitem{B} V. Blomer, {\em Subconvexity for a double Dirichlet series}, Compos. Math. {\bf 147} (2011), 355--374.

\bibitem{C} H. Cohen, {\em Sums involving the values at negative integers of $L$-functions of quadratic characters}, Math. Ann. {\bf 217} (1975), 271--285.

\bibitem{Datskovsky} B. Datskovsky, {\em A mean-value theorem for class numbers of quadratic extensions}, in: A tribute to Emil Grosswald: number theory and related analysis, 179--242, Contemp. Math. {\bf 143}, Amer. Math. Soc., Providence, RI, 1993.

\bibitem{DG} N. Diamantis and D. Goldfeld, {\em A converse theorem for double Dirichlet series and Shintani zeta functions}, J. Math. Soc. Japan {\bf 66} (2014), 449--477. 

\bibitem{FF} B. Fisher and S. Friedberg, {\em Double Dirichlet series over function fields}, Compos. Math. {\bf 140} (2004), 613--630.

\bibitem{GH} D. Goldfeld and J. Hoffstein, {\em Eisenstein series of $1/2$-integral weight and the mean value of real Dirichlet $L$-series}, Invent. math. {\bf 80} (1985), 185--208.

\bibitem{HW} W. Hoffmann and S. Wakatsuki, {\em On the geometric side of the Arthur trace formula for the symplectic group of rank 2}, Mem. Amer. Math. Soc. {\bf 255} (2018), no. 1224. 

\bibitem{H} L. H\"ormander, An introduction to complex analysis in several variables, Van Nostrand, Princeton, NJ, 1966.

\bibitem{IS1} T. Ibukiyama and H. Saito, {\em On zeta functions associated to symmetric matrices I: An Explicit Form of Zeta Functions}, Amer. J. Math. {\bf 117} (1995), 1097--1155. 

\bibitem{IS2} \bysame, {\em On zeta functions associated to symmetric matrices II: Functional equations and special values}, Nagoya Math. J. {\bf 208} (2012), 263--315. 

\bibitem{Igusa} J. Igusa, An introduction to the theory of local zeta functions. AMS/IP Studies in Advanced Mathematics, 14. American Mathematical Society, Providence, RI; International Press, Cambridge, MA, 2000.

\bibitem{Ik} T. Ikeda, {\em On the functional equation of the Siegel series}, J. Number Theory {\bf 172} (2017), 44--62.

\bibitem{J} M. Jutila, {\em On the mean value of $L(\frac 12,\chi)$ for real characters}, Analysis {\bf 1} (1981), no. 2, 149--161. 
\bibitem{KWY} H. Kim, S. Wakatsuki, and T. Yamauchi, {\em Equidistribution theorem of holomorphic Siegel cusp forms of general degree}, in preparation.

\bibitem{Mu} W. M\"uller, {\em On the singularities of residual intertwining operators}, Geom. Funct. Anal. {\bf 10} (2000), 1118--1170.

\bibitem{Saito}H. Saito, {\em Explicit formula of orbital $p$-adic zeta functions associated to symmetric and Hermitian matrices}, Comment. Math. Univ. St. Paul. {\bf 46} (1997), 175--216. 

\bibitem{Saito2}\bysame, {\em Explicit form of the zeta functions of prehomogeneous vector spaces}, Math. Ann. {\bf 315} (1999), 587--615.


\bibitem{Sato}F. Sato, {\em Zeta functions in several variables associated with prehomogeneous vector spaces, I. Functional equations}, Tohoku Math. J. (2) {\bf 34} (1982), 437--483.

\bibitem{Sato2}\bysame, {\em On zeta functions of ternary zero forms}, J. Fac. Sci. Univ. Tokyo. Sect. 1 A {\bf 28} (1982), pp. 585--604.

\bibitem{Sato3}\bysame, {\em On functional equations of zeta distributions}, Automorphic forms and geometry of arithmetic varieties, 465--508, Adv. Stud. Pure Math., 15, Academic Press, Boston, MA, 1989.

\bibitem{SS} M. Sato and T. Shintani, {\em On zeta functions associated with prehomogeneous vector spaces}, Ann. of Math. (2) {\bf 100} (1974), 131--170. 

\bibitem{Shintani} T. Shintani, {\em On zeta functions associated with the vector space of quadratic forms}, J. Fac. Sci. Univ. Tokyo Sect. IA Math. {\bf 22} (1975), 25--65.

\bibitem{Taniguchi} T. Taniguchi, Distributions of discriminants of cubic algebras, math.NT/0606109, 2006.
\bibitem{Titch} E.C. Titchmarsh, The theory of the Riemann zeta-function. Second edition. Edited and with a preface by D. R. Heath-Brown. The Clarendon Press, Oxford University Press, New York, 1986.

\bibitem{Wakatsuki} S. Wakatsuki, {\em The dimensions of spaces of Siegel cusp forms of general degree}, Adv. Math. {\bf 340} (2018), 1012--1066.

\bibitem{Wen} J. Wen, {\em Bhargava Integer Cubes and Weyl Group Multiple Dirichlet Series}, arXiv:1311.2132, 2015.

\bibitem{W} D. Wright, {\em The adelic zeta function associated to the space of binary cubic forms. I. Global theory}, Math. Ann. {\bf 270} (1985), no. 4, 503--534.

\bibitem{Young} M.P. Young, The first moment of quadratic Dirichlet $L$-functions, Acta Arith. {\bf 138} (2009), 73--99.

\bibitem{Yukie} A. Yukie, {\em On the Shintani zeta function for the space of binary quadratic forms}, Math. Ann. {\bf 292} (1992), 355--374. 

\end{thebibliography}
\end{document}